\documentclass[11pt]{article}

\usepackage{cite}
\usepackage{graphicx,color}
\usepackage[normalem]{ulem}
\usepackage{amsmath,amssymb,amsthm}
\usepackage{ifthen}
\usepackage{mathrsfs}
\usepackage{mathtools}
\usepackage{enumitem}
\usepackage{lineno}
\usepackage{float}
\usepackage{fancyhdr}
\usepackage[us,12hr]{datetime}
\usepackage{exscale}
\usepackage{tabularx}
\usepackage{latexsym}
\usepackage{fullpage}
\usepackage{verbatim}
\usepackage{multirow}
\usepackage{comment}
\usepackage{hyperref}
\usepackage{adjustbox}
\usepackage[noabbrev]{cleveref}
\usepackage{tikz}
\usepackage{pdflscape}
\usepackage{algorithm}
\usepackage{algpseudocode}
\usepackage{authblk}

\newtheorem{lemma}{Lemma}[section]

\newtheorem{remark}{Remark}[section]

\crefname{theorem}{Theorem}{Theorems}
\Crefname{theorem}{Theorem}{Theorems}
\crefname{proposition}{Proposition}{Propositions}
\Crefname{proposition}{Proposition}{Propositions}
\crefname{corollary}{Corollary}{Corollaries}
\Crefname{corollary}{Corollary}{Corollaries}
\crefname{lemma}{Lemma}{Lemmas}
\Crefname{lemma}{Lemma}{Lemmas}
\crefname{remark}{Remark}{Remarks}
\Crefname{remark}{Remark}{Remarks}
\crefname{equation}{}{}
\Crefname{equation}{}{}

\def\eqref#1{{\normalfont(\ref{#1})}}

\DeclareMathOperator{\aff}{aff}
\DeclareMathOperator{\range}{range}

\DeclareMathOperator{\ri}{ri}

\DeclareMathOperator{\conv}{conv}

\DeclareMathOperator{\arrow}{arrow}

\newcommand{\Rnn}{\mathbb{R}^{n}}

\newcommand{\Sn}{\mathbb{S}^{n}}
\newcommand{\Sno}{\mathbb{S}^{n+1}}
\newcommand{\Snop}{\mathbb{S}_{+}^{n+1}}
\newcommand{\Snp}{\mathbb{S}_{+}^{n}}

\newcommand{\rvv}[1]{#1}

\author{%
Hao Hu\thanks{School of Mathematical and Statistical Sciences, Clemson University, Clemson, SC, USA; Email: \url{hhu2@clemson.edu}; Research supported by the Air Force Office of Scientific Research under award number FA9550-23-1-0508.}
\qquad
Mingming Xu\thanks{School of Mathematical and Statistical Sciences, Clemson University, Clemson, SC, USA; Email: \url{mingmix@g.clemson.edu}; Research supported by the Air Force Office of Scientific Research under award number FA9550-23-1-0508.}
}

\begin{document}

\title{A Primal Approach to Facial Reduction for SDP Relaxations of Combinatorial Optimization Problems}
\break
\date{\currenttime, \today
}
\maketitle

\medskip

\begin{abstract}
We propose a novel facial reduction algorithm tailored to semidefinite programming relaxations of combinatorial optimization problems with quadratic objective functions. Our method leverages the specific structure of these relaxations, particularly the availability of feasible solutions that can often be generated efficiently in practice. By incorporating such solutions into the facial reduction process, we substantially simplify the reduction steps. On average, our facial reduction algorithm is four times faster than the standard implementation on the considered benchmark sets, providing significantly improved preprocessing for SDP relaxations in combinatorial optimization.
\end{abstract}


\section{Introduction}

Combinatorial optimization problems are central to operations research, computer science, and engineering, as they model a wide range of decision-making tasks involving discrete choices, such as routing, scheduling, and resource allocation. Traditionally, many of these problems have been studied with linear objective functions due to their mathematical tractability and the availability of well-established solution methods. However, extending the objective from linear to quadratic allows for modeling more complex interactions between decision variables, capturing dependencies that linear models cannot. 

A prominent example is the Quadratic Assignment Problem (QAP), see, \cite{burkard1984quadratic,finke1987quadratic}, that arises in facility layout planning. In QAP, the cost depends not only on individual assignments but also on the interaction between pairs of facilities and their respective locations. This quadratic formulation provides a more accurate and realistic representation of practical problems, albeit at the cost of increased computational complexity.

Such problems can be formulated as a \emph{mixed-binary quadratic programming (MBQP)} problem. Let \( Q \in \mathbb{R}^{n \times n} \) be a symmetric\footnote{Throughout this paper, we assume $Q$ is symmetric without loss of generality. If $Q$ is not symmetric, it can be replaced by its symmetric part $\frac{1}{2}(Q + Q^\top)$ without altering the objective function value.} cost matrix and \( c \in \mathbb{R}^n \) a linear cost vector. The MBQP is defined as:
\begin{equation}\label{bqp}
	\min \left\{  x^\top Q x + c^{\top}x \mid x \in P \right\},
\end{equation}
where the feasible region \( P \) is a mixed-binary set defined by the following constraints:
\begin{equation}\label{bqpf}
	\begin{array}{rlll}
		a_{i}^{\top}x &=& b_{i} &\text{ for } i = 1, \ldots, p, \\
		a_{i}^{\top}x &\leq& b_{i} &\text{ for } i = p+1, \ldots, m, \\
		x_{i} &\in& \{0,1\} &\text{ for } i = 1, \ldots, r,
	\end{array}
\end{equation}
where \( a_{i} \in \mathbb{R}^n \) and \( b_{i} \in \mathbb{R} \) for \( i = 1, \ldots, m \).

One of the most successful approaches for solving \eqref{bqp} is via \emph{semidefinite programming (SDP)} relaxations, which provide strong convex approximations to otherwise hard combinatorial problems. Unlike traditional linear programming (LP) relaxations, SDP relaxations capture quadratic relationships and global structural properties, often leading to tighter bounds and higher-quality solutions when combined with rounding heuristics or used within exact algorithms. Landmark results, such as the Goemans–Williamson algorithm for Max-Cut \cite{goemans1995improved}, illustrate how SDP can yield approximation guarantees superior to those of LP-based methods. Furthermore, SDP formulations naturally appear in a variety of application domains and have achieved very good computational results, including graph partitioning \cite{wolkowicz1999semidefinite,wiegele2022sdp}, sensor network localization \cite{biswas2006semidefinite,kim2009exploiting,krislock2010explicit}, polynomial optimization \cite{anjos2011handbook,waki2006sums,de2011lasserre,papp2017semi,hu2016note,jibetean2005semidefinite} and quantum information theory \cite{hu2022robust,fawzi2018efficient,fawzi2023optimal,tavakoli2024semidefinite,karimi2025efficient,araujo2023quantum}, further demonstrating their theoretical and practical relevance. The continuing development of efficient SDP solvers and scalable algorithms has made SDP a powerful tool in tackling complex combinatorial optimization problems.

Despite their advantages, SDP relaxations are computationally intensive and highly sensitive to problem regularity, particularly the presence of strict feasibility. In particular, the failure of \emph{Slater’s condition}—the existence of a strictly feasible solution—can severely impair both the theoretical guarantees and numerical performance of SDP relaxations. In the absence of Slater’s condition, the dual problem may exhibit a nonzero duality gap, and primal and dual optimal values may not coincide. This can lead to weaker bounds and unreliable results. Furthermore, numerical solvers may encounter instability or fail to converge. 

A standard remedy is \emph{facial reduction}, a regularization procedure introduced in~\cite{borwein1981regularizing,borwein1981facial}, which iteratively reformulates the SDP to ensure strict feasibility. While facial reduction guarantees regularity, it typically requires solving a sequence of SDP auxiliary problems that may be as computationally demanding as the original problem. To improve practical applicability, several special \emph{facial reduction algorithms} (FRAs) have been developed; see, e.g., \cite{cheung2013preprocessing,friberg2016facial,permenter2018partial,zhu2019sieve,hu2023affine}.

In this paper, we propose a novel FRA specifically designed for SDP relaxations of MBQP problems. Our key insight is that, while MBQPs are generally difficult to solve, it is often straightforward to generate feasible solutions within the mixed-binary feasible set~\( P \) in a structured manner. For instance, although the QAP is well-known for its computational difficulty, its feasible region consists of binary assignment matrices, which are simple to construct. These feasible solutions can be used to generate corresponding feasible matrix solutions for the SDP relaxation, thereby substantially simplifying the facial reduction process. We call the resulting method \emph{primal FRA} to emphasize its key feature---leveraging the primal feasible solutions. The primal FRA is simple to implement and numerically robust in practice. We provide both theoretical and computational results related to the primal FRA in this paper. Numerical experiments demonstrate that our method substantially reduces computation time on nearly all benchmark instances and frequently restores Slater’s condition without the need to solve additional SDP auxiliary problems.

\textbf{Notations.} For any set \( P \), we denote its linear span by \( \operatorname{span}(P) \), its affine hull by \( \operatorname{aff}(P)\), its relative interior by \( \operatorname{ri}(P) \), and its convex hull by \( \operatorname{conv}(P) \). The all-zeros vector and matrix are denoted by \( \mathbf{0} \) and \( \mathbf{O} \), respectively; their dimensions will be clear from the context. If \( K \) is a closed convex cone, its \emph{dual cone} is defined as
$
K^{*} := \left\{ y \;\middle|\; \langle y, x \rangle \geq 0,\; \forall x \in K \right\}.
$
The set of \( n \times n \) symmetric matrices is denoted by \( \mathbb{S}^n \). For \( X, Y \in \mathbb{S}^n \), the trace inner product is defined as \( \langle X, Y \rangle := \operatorname{tr}(XY) \). The set of \( n \times n \) symmetric positive semidefinite matrices is denoted by \( \mathbb{S}^n_{+} \), and the set of symmetric positive definite matrices is denoted by \( \mathbb{S}^n_{++} \). The range space and the null space of a given matrix $M \in \Sn$ are denoted by $\operatorname{range}(M)$ and $\operatorname{null}(M)$, respectively.

\section{Preliminaries}

\subsection{Facial Reduction}
In this section, we review the theory of facial reduction. To illustrate the key idea of facial reduction for general readers, consider a simple example where $K = \mathbb{S}_{+}^{2}$ and $L=\{ X \in \mathbb{S}^{2} \mid X_{22} = 0 \}$.
	The intersection $L \cap \mathbb{S}_{+}^{2}$ consists of positive semidefinite matrices with a zero on the diagonal. Since any $X \in \mathbb{S}_{+}^{2}$ with $X_{ii}=0$ must have the entire $i$-th row and column as zero, the feasible set forces $X_{12} = X_{21} = 0$.
	Consequently, this set does not contain any matrix in the interior of the cone $\mathbb{S}_{+}^{2}$ (where $X$ must be positive definite). This lack of strictly feasible points, known as the failure of Slater's condition, often leads to numerical instability and theoretical difficulties in optimization.
	
	However, we can define a smaller cone: $
	F := \{ X \in \mathbb{S}_{+}^{2} \mid X_{12} = X_{22} = 0 \}.$
	Note that $F$ is a proper \emph{face} of $\mathbb{S}_{+}^{2}$, and importantly, $L \cap \mathbb{S}_{+}^{2} = L \cap F$.
	Unlike the original formulation, the set $L \cap F$ contains points like $[\begin{smallmatrix} 1 & 0 \\ 0 & 0 \end{smallmatrix}]$ which lie in the relative interior of the cone $F$. By replacing the ambient cone $\mathbb{S}_{+}^{2}$ with the smaller face $F$, we obtain a regularized formulation that is more efficient and numerically stable. We call this process facial reduction.
	
Let $K$ be a closed convex cone.  We say $F$ is a \emph{face} of $K$, denoted by $F \unlhd K$, if $x,y \in K$ and $x+y \in F$ imply that $x,y \in F$.  The conjugate face of $F$ is $F^{\triangle} := K^{*} \cap F^{\perp}.$ A face $F$ of $K$ is called \emph{exposed} if it is of the form $F = K \cap w^{\perp}$ for some $w \in K^{*}$. The element $w$ is then called an \emph{exposing vector}. We say $K$ is \emph{exposed} if all of its faces are exposed.   For example, the cone $\Snp$ is exposed; the non-empty faces of $\Snp$ can be characterized by linear subspaces: $F$ is a non-empty face of $\Snp$ if and only if there exists a linear subspace $\mathcal{V} \subseteq \Rnn$ such that
$F =  \left\{ X \in \Snp \mid  \range(X) \subseteq \mathcal{V} \right\}.$

Let $L$ be an affine subspace such that $L \cap K \neq \emptyset$. We say \emph{Slater's condition} holds for $L \cap K$ if $L \cap \ri (K) \neq \emptyset$, i.e., it contains a feasible solution in the relative interior of $K$. Facial reduction exploits the following theorem of the alternative:
$L \cap \ri (K) = \emptyset$ if and only if $L^{\perp}  \cap (K^{*} \setminus K^{\perp}) \neq \emptyset.$\footnote{For an affine subspace $L = \{ X \mid \mathcal{A}(X) = b \}$, we define $L^{\perp} := \{ W \mid \langle W, X \rangle = 0, \forall X \in L \}$. This corresponds to vectors $W = \mathcal{A}^*(y)$ where $b^\top y = 0$.}
Thus, $L \cap K$ fails Slater's condition if and only if there exists $w \in L^{\perp} \cap \Snp \neq \emptyset$. When $K = \Snp$, $w$ is an exposing vector of $\Snp$, and $F = \Snp \cap w^{\perp}$ is a proper face of $\Snp$ containing $L \cap \Snp$. Thus, if $L \cap \Snp$ fails Slater's condition, then we can replace the original feasible set $L \cap \Snp$ with the equivalent set $L \cap F$, where $F$ is a proper face of $\Snp$. This procedure can be iterated for a finite number of times until $L \cap F$ satisfies Slater's condition. We outline FRA in Algorithm~\ref{alg1} below.

\begin{algorithm}[H]
	\caption{Facial Reduction Algorithm (FRA)}\label{alg1}
	\begin{algorithmic}[1]
		\State \textbf{Inputs:} An affine subspace $L$ in $\Sn$ and a face $K$ of $\Snp$.
		\State \textbf{Initialization:} Let $F_0 = K$, $i = 1$.
		\While{we can pick $w_i \in L^\perp \cap (F_{i-1}^{*} \setminus F_{i-1}^\perp)$}
		\State Set $F_i \leftarrow F_{i-1} \cap w_i^\perp$.
		\State Set $i \leftarrow i + 1$.
		\EndWhile
	\end{algorithmic}
\end{algorithm}

The finite convergence of FRA has also been established in several recent works; see~\cite{liu2018exact,pataki2013strong,waki2013facial}.  
While FRA offers a systematic framework for restoring Slater’s condition for the intersection \( L \cap K \), it faces a significant computational challenge: at each iteration, it requires solving the following problem, $w_i \in L^\perp \cap \left(F_{i-1}^{*} \setminus F_{i-1}^\perp\right),$
which is itself a nontrivial optimization task. We refer to the problem of identifying such an element \( w_i \) as the \emph{FR auxiliary problem}. For example, in the context of SDP, when applying FRA to \( L \cap \mathbb{S}_+^n \), the FR auxiliary problem is to find a nonzero matrix in \( L^{\perp} \cap \mathbb{S}_+^n \), which translates into solving an SDP feasibility problem. This implies that the FR auxiliary problem can be as difficult as the original SDP itself. Consequently, FRA may become impractical for large-scale problems or when applied directly to SDP relaxations of combinatorial optimization models.

\subsection{SDP Relaxations}

Consider the feasible set \( P \) defined in \eqref{bqpf}. We study SDP relaxations of the following form. By lifting feasible solutions into \( \mathbb{S}^{n+1} \), we define:
\begin{equation}\label{liftedP}
	S := \left\{ \begin{bmatrix}
		1\\
		x
	\end{bmatrix}
	\begin{bmatrix}
		1 & x^{\top}
	\end{bmatrix} \,\middle|\, x \in P \right\}.
\end{equation}
Note that all matrices in \( S \) are positive semidefinite. For any polyhedral set \( L \subseteq \mathbb{S}^{n+1} \) such that \( S \subseteq L \), the intersection
\begin{equation}\label{sdprelax}
	L \cap \mathbb{S}_+^{n+1}
\end{equation}
forms an outer approximation to \( S \) and is referred to as an \emph{SDP relaxation} of \( S \). Since there is a one-to-one correspondence between the original feasible set \( P \) and its lifted counterpart \( S \), we also refer to \( L \cap \mathbb{S}_+^{n+1} \) as an SDP relaxation of \( P \). One of the most well-known SDP relaxation is \emph{Shor's relaxation}, the specific details of which we provide later in Section 5.

We can efficiently optimize a linear function over this set, see \cite{nesterov1994interior,wolkowicz2012handbook,grotschel2012geometric}, yielding a lower bound for the original MBQP problem \eqref{bqp}:
\begin{equation}\label{sdprelaxopt}
	\begin{array}{rll}
		\min & \left\{  \langle \tilde{Q}, Y \rangle \mid  Y \in L \cap \mathbb{S}_+^{n+1} \right\}\\
	\end{array}
\end{equation}
where
$
\tilde{Q} := [\begin{smallmatrix}
	0 & \frac{1}{2}c^{\top} \\
	\frac{1}{2}c & Q
\end{smallmatrix}] \in \mathbb{S}^{n+1}.
$
We describe the so-called \emph{arrow constraints} which are implied by the binary variables. The rows and columns of the matrices in $S$ are indexed by $\{0,1,\ldots,n\}$. For any \( Y \in S \), the binary constraints on \( x \) imply \( Y_{00} = 1 \) and \( Y_{0i} = Y_{ii} \) for \( i = 1, \ldots, r \). Define the arrow operator \( \arrow: \mathbb{S}^{n+1} \to \mathbb{R}^{r+1} \) as:
\begin{equation}\label{def_arrow}
	\arrow(Y) := \begin{bmatrix}
		Y_{00} \\
		Y_{11} - \frac{1}{2}(Y_{01} + Y_{10}) \\
		\vdots \\
		Y_{rr} - \frac{1}{2}(Y_{0r} + Y_{r0})
	\end{bmatrix}.
\end{equation}
Let \( e_0 \) be the standard unit vector in \( \mathbb{R}^{r+1} \) with 1 in the first coordinate. The constraint \( \arrow(Y) = e_0 \), known as the arrow constraint, holds for all \( Y \in S \). Note that the arrow operator depends on the number \( r \) of binary variables, though we omit the superscript for simplicity. In Section~\ref{sec_num}, we will describe three SDP relaxations of varying strength and computational cost, all of which will be used to test our special FRA.

To simplify the presentation of applying facial reduction to the SDP relaxation \( L \cap \Snop \), we assume throughout the theoretical discussion that \( L \) is an affine set.
Under this assumption, the FR auxiliary problem involves the same set \( L^{\perp} \cap \Snop \) as in Algorithm~\ref{alg1}. If $L \cap \Snop$ does not satisfy Slater's condition, applying the first iteration of Algorithm~\ref{alg1} to the SDP relaxation \( L \cap \Snop \), we obtain a proper face \( F \unlhd \Snop \) such that \( L \cap F = L \cap \Snop \). Since \( F \) can have lower dimension than \( \Snop \), the reformulated problem \( L \cap F \) constitutes a smaller SDP relaxation. 

For any face \( F \unlhd \Snop \), we refer to \( L \cap F \) as a \emph{facially reduced formulation} of \( L \cap \Snop \) if
$
S \subseteq L \cap F \subseteq L \cap \Snop.
$
This definition slightly generalizes the standard FRA described in Algorithm~\ref{alg1}, which only produces faces \( F \unlhd \Snop \) satisfying \( L \cap F = L \cap \Snop \). The generalized notion is particularly useful for unifying various specialized FRA approaches under a common framework.

\begin{remark}[Linear inequalities in computational formulations]\label{rem:ineq}
	\rvv{The computational formulations in Section~\ref{sec_num} additionally contain linear inequality constraints, so that \( L \) takes the polyhedral form \( L = \{ Y \in \Sno \mid \mathcal{A}(Y) = b, \ \mathcal{B}(Y) \leq d \} \) rather than an affine subspace. Here \( \mathcal{A} \) is the linear map collecting the equality constraints (as in the footnote above) and \( \mathcal{B} \) is the linear map collecting the inequality constraints, with adjoints \( \mathcal{A}^{*} \) and \( \mathcal{B}^{*} \). Such inequalities are readily accommodated. Introducing nonnegative slack variables turns them into equalities and replaces the cone \( \Snop \) with \( \Snop \times \mathbb{R}^{k}_{+} \), rendering \( L \) affine in the extended space. Equivalently, one may retain the inequalities and only adjust the FR auxiliary problem, seeking an exposing certificate of the form \( W = \mathcal{A}^{*}(y) + \mathcal{B}^{*}(z) \), where \( y \) is free for the equality constraints and \( z \geq 0 \) for the inequality constraints. The only change is thus the additional sign constraint \( z \geq 0 \) on the multipliers; we no longer have the simple expression \( L^{\perp} \cap \Snop \) for the auxiliary feasible set, but the facial reduction procedure is otherwise identical. We note, moreover, that this subtlety does not affect the primal FRA developed in Section~3 in its main mode of operation, where a maximum-rank feasible solution yields \( L \cap G \) directly without solving any FR auxiliary problem.}
\end{remark}

\section{A Primal Approach to Facial Reduction}

\subsection{The Primal FRA}

Standard implementations of FRA require finding a nonzero element in \( L^{\perp} \cap \mathbb{S}_+^n \) which is itself a computationally demanding task. (Note that we use $\Snp$ for general discussions, and $\Snop$ when referring to SDP relaxations.) To address this challenge, several practical strategies have been proposed. One such approach is to replace \( \mathbb{S}_+^n \) with a tractable inner approximation \( K \), and instead search for a point in the subset \( L^{\perp} \cap K \). This idea, known as \emph{partial facial reduction}, was introduced in~\cite{permenter2018partial}. Despite variations in implementation, the essential goal of all FRAs remains the same: to reduce the computational burden associated with solving the FR auxiliary problem.

The method proposed in this work follows this principle but adopts a novel strategy. Our key observation is that any primal feasible solution exposes a face of \( \mathbb{S}_+^n \) that contains the entire set \( L^{\perp} \cap \mathbb{S}_+^n \); see Lemma~\ref{smallface} and~\ref{smallface2}. This observation enables us to reduce the dimension of \( L^{\perp} \cap \mathbb{S}_+^n \). Specifically, we use a primal feasible solution to define a smaller face that still contains all exposing vectors in $L^{\perp} \cap \Snp$. This leads to the following special FRA:
\begin{algorithm}[H]
	\caption{Primal Facial Reduction Algorithm (Primal FRA)}
	\label{alg2}
	\begin{algorithmic}[1]
		\Statex \textbf{Input:} An SDP relaxation \( L \cap \mathbb{S}_+^n \) for $S$ in \eqref{liftedP}, and a feasible solution \( X^* \in L \cap \mathbb{S}_+^n \).
		\Statex \textbf{Output:} A facially reduced formulation of \( L \cap \mathbb{S}_+^n \).
		\State Define the face $G \unlhd \Snp$,
		\begin{equation}\label{newG}
			G := \left\{ X \in \mathbb{S}_+^n \mid \operatorname{range}(X) \subseteq \operatorname{range}(X^*) \right\}.
		\end{equation}
		\If{\( X^* \) has maximum rank in \( \conv(S) \)}
		\State \Return \( L \cap G \)
		\Else
		\State Find an element \( W^* \in L^\perp \cap G^{\triangle} \). Define the face
		\begin{equation}\label{newF}
			F := \left\{ X \in \mathbb{S}_+^n \mid \operatorname{range}(X) \subseteq \operatorname{null}(W^*) \right\}.
		\end{equation}
		\State \Return \( L \cap F \)
		\EndIf
	\end{algorithmic}
\end{algorithm}

The description of Algorithm~\ref{alg2} naturally raises several questions: How can a feasible solution be obtained? How can we determine whether it has maximum rank? These practical concerns are addressed in detail in Section~\ref{genfea}, but for now, we focus on the correctness of the algorithm. In particular, we show that both \( L \cap G \) and \( L \cap F \) are facially reduced formulations of the original feasible set \( L \cap \mathbb{S}_+^n \).

\begin{lemma}\label{smallface}
	In Algorithm~\ref{alg2}, the face \( G \) defined in \eqref{newG} satisfies the following.
	\begin{enumerate}
		\item \rvv{Without any additional assumption,} \( L \cap G \subseteq L \cap \mathbb{S}_+^n \) and \( L \cap G \) satisfies Slater's condition.
		\item \rvv{If, in addition, \( X^* \) has maximum rank in \( \conv(S) \), then} \( S \subseteq L \cap G \).
	\end{enumerate}
\end{lemma}

\begin{proof}
	\rvv{We prove the two parts separately.}

	\rvv{\emph{Part 1.}} By construction, \( G \) is a face of \( \mathbb{S}_+^n \), so \( L \cap G \subseteq L \cap \mathbb{S}_+^n \). \rvv{Since \( X^* \in L \) and \( X^* \in \operatorname{ri}(G) \) by the definition of \( G \) in \eqref{newG}, the point \( X^* \) lies in \( L \cap \operatorname{ri}(G) \). Hence \( L \cap G \) satisfies Slater's condition. Note that this part does not rely on any rank assumption on \( X^* \).}

	\rvv{\emph{Part 2.} Now suppose \( X^* \) has maximum rank in \( \conv(S) \).} To show the inclusion $S \subseteq L \cap G$, let \( X \in S \) be arbitrary. \rvv{Since \( X^* \) has maximum rank in \( \conv(S) \),} we must have \( \operatorname{range}(X) \subseteq \operatorname{range}(X^*) \), or else the average \( \tfrac{1}{2}(X + X^*) \) would also be in $\conv(S)$ and have strictly higher rank than \( X^* \), contradicting its maximality.
	Hence, \( X \in G \), and since \( S \subseteq L \), we obtain \( X \in L \cap G \). This proves that \( S \subseteq L \cap G\).
\end{proof}

\begin{lemma}\label{smallface2}
	In Algorithm~\ref{alg2}, we have \( L \cap F = L \cap \mathbb{S}_+^n \).
\end{lemma}

\begin{proof}
	It suffices to show that \( L^{\perp} \cap \mathbb{S}_+^n = L^{\perp} \cap G^{\triangle} \). Since \( G^{\triangle} \unlhd \mathbb{S}_+^n \), we clearly have \( L^{\perp} \cap \mathbb{S}_+^n \supseteq L^{\perp} \cap G^{\triangle} \). For the reverse inclusion, let \( W \in L^{\perp} \cap \mathbb{S}_+^n \). Then \( \langle W, X \rangle = 0 \) for all \( X \in L \cap \mathbb{S}_+^n \), and in particular for \( X^* \). Therefore, \( \operatorname{range}(W) \subseteq \operatorname{null}(X^*) \), which implies \( W \in G^{\triangle} \). Thus, \( W \in L^{\perp} \cap G^{\triangle} \), completing the proof.
\end{proof}

We highlight some key features of the primal FRA:

\begin{enumerate}
	\item \textbf{[Avoiding SDP Solves]} If the algorithm returns \( L \cap G \), then an SDP solve is completely avoided, and Slater's condition is guaranteed to hold.
	
	\item \textbf{[Reduced FR Auxiliary Problem]} If the rank of \( X^* \) is high, then the dimension of \( L^{\perp} \cap G^{\triangle} \) is small. This makes it computationally inexpensive to search for an exposing vector in the set \( L^{\perp} \cap G^{\triangle} \), which we refer to as the \emph{reduced FR auxiliary problem}.
	
	\item \textbf{[Compatibility with Existing Methods]} If the rank of \( X^* \) is low, finding a non-trivial element in \( L^{\perp} \cap G^{\triangle} \) remains challenging. However, we can apply existing methods such as \emph{partial FRA} \cite{permenter2018partial} or \emph{Sieve-SDP} \cite{zhu2019sieve} to this smaller set \( L^{\perp} \cap G^{\triangle} \), which can still lead to computational savings.

\end{enumerate}

\subsection{Generating Feasible Solutions}\label{genfea}

While finding a feasible solution for a general SDP can be challenging, it is often relatively straightforward in the context of SDP relaxations arising from combinatorial optimization problems. 

Let \( P \) denote the feasible set as defined in \eqref{bqpf}, and let \( L \cap \mathbb{S}_+^{n+1} \) be its SDP relaxation as defined in \eqref{sdprelax}. In view of the lifted formulation~\eqref{liftedP}, it suffices to identify a feasible solution \( x \in P \); the corresponding lifted matrix $[\begin{smallmatrix}
	1\\
	x
\end{smallmatrix}]
[\begin{smallmatrix}
	1 & x^{\top}
\end{smallmatrix}]$
is then a feasible solution in \( L \cap \mathbb{S}_+^{n+1} \).

Naturally, it is desirable to construct a feasible matrix with higher rank in order to simplify the set $L^{\perp} \cap \Snop$. To achieve this, we may generate a collection of feasible solutions \( x_0, \ldots, x_k \in P \) and define
\begin{equation}\label{Xstar}
	X^* := \frac{1}{k+1}\sum_{i=0}^k
	\begin{bmatrix}
		1\\
		x_i
	\end{bmatrix}
	\begin{bmatrix}
		1 & x_i^{\top}
	\end{bmatrix}.
\end{equation}
It is straightforward to verify that the vectors \( x_0, \ldots, x_k \) are affinely independent if and only if \( \operatorname{rank}(X^*) = k+1 \). Thus, our goal is to generate as many affinely independent feasible solutions in \( P \) as possible. In particular, if a maximal affinely independent subset of \( P \) is found, then \( X^* \) attains the maximum possible rank among all matrices in \( \conv(S) \). This connection between the affine structure of the original feasible set and the maximum-rank feasible matrix in $\conv(S)$ was first analyzed in \cite{tunccel2001slater}.

We now describe a method for finding a maximal affinely independent subset of feasible solutions in \( P \). Suppose that \( x_0, \ldots, x_k \in P \) are known to be affinely independent. We can either certify that they span the affine hull of \( P \), or identify a new feasible solution \( x \in P \) that is affinely independent of \( x_0, \ldots, x_k \). By iterating this procedure, we can incrementally construct a maximal affinely independent subset of \( P \), which in turn yields a maximal-rank matrix \( X^* \) for use in Algorithm~\ref{alg2}.

\begin{lemma}\label{main_thm}
	Let \( v \in P \), and let \( H \) be a linear subspace such that \( v + H \subseteq \operatorname{aff}(P) \). Then \( v + H = \operatorname{aff}(P) \) if and only if
	\begin{equation}\label{maineq}
		u^\top (x - v) = 0 \quad \text{for all } x \in P \text{ and } u \in H^\perp.
	\end{equation}
\end{lemma}
\begin{proof}
	Assume \( v + H = \operatorname{aff}(P) \). Then for any \( x \in P \), we have \( x - v \in H \), which implies \( u^\top (x - v) = 0 \) for all \( u \in H^\perp \), since every vector in \( H^\perp \) is orthogonal to \( H \). Conversely, suppose \( v + H \subsetneq \operatorname{aff}(P) \). Then there exists \( x \in P \) such that \( x - v \notin H \). By the fundamental theorem of linear algebra, there exists \( u \in H^\perp \) such that \( u^\top (x - v) \neq 0 \). This shows that \eqref{maineq} does not hold, completing the proof.
\end{proof}

We can use the characterization of the affine hull of \( P \) provided in Lemma~\ref{main_thm} to compute \( \operatorname{aff}(P) \) as follows. 
Assume \( v + H \subseteq \operatorname{aff}(P) \) for some \( v \in P \) and a \((k - 1)\)-dimensional linear subspace \( H \). Let \( u_1, \ldots, u_{n - k + 1} \) be any linearly independent vectors spanning \( H^\perp \). Then condition~\eqref{maineq} holds if and only if
\begin{equation}\label{maxmin}
	\max \{ u_i^\top x \mid x \in P \} = \min \{ u_i^\top x \mid x \in P \}, \quad \text{for } i = 1, \ldots, n - k + 1.
\end{equation}
Thus, we can verify whether~\eqref{maineq} holds by solving at most \( 2(n - k + 1) \) optimization problems of the form~\eqref{maxmin}. Moreover, if equality fails for any index \( i \), then either the maximizer or minimizer, say \( x^* \), satisfies \( u_i^\top (x^* - v) \neq 0 \). This implies \( x^* - v \notin H \), and by Lemma~\ref{main_thm}, we can construct a strictly larger subspace \( \tilde{H} = \operatorname{span}(H \cup \{x^* - v\}) \), for which \( v + \tilde{H} \subseteq \operatorname{aff}(P) \).

Starting with an initial feasible solution \( v \in P \) and \( H = \{0\} \), we iteratively construct a set of affinely independent feasible points that span \( \operatorname{aff}(P) \). This process terminates in at most \( n \) iterations. In practice, however, we may terminate the algorithm earlier due to computational limitations encountered while solving the subproblems in \eqref{maxmin}. The full procedure is summarized in Algorithm~\ref{alg3}.

\begin{algorithm}[H]
	\caption{Generating a Feasible Matrix Solution}\label{alg3}
	\begin{algorithmic}[1]
		\Statex \textbf{Input:} A mixed-binary feasible set \( P \) as defined in \eqref{bqpf}.
		\Statex \textbf{Output:} A feasible matrix solution \( X^* \in \conv(S) \), see \eqref{liftedP}.
		\State Find \( v:=x_0 \in P \) and set \( H = \{0\} \).
		\For{$k = 1, \ldots, n$}
		\State Let \( u_1, \ldots, u_{n-k+1} \) be linearly independent vectors spanning \( H^{\perp} \).
		\If{we can find \( x_k \in P \setminus (v + H) \) via \eqref{maxmin}}
		\State \( H \leftarrow \operatorname{span}(H \cup \{x_k - v\}) \).
		\ElsIf{we can verify \( v + H = \operatorname{aff}(P) \) via \eqref{maxmin}}
		\State \Return \( X^* \) as defined in \eqref{Xstar}; it has maximum rank.
		\ElsIf{some \eqref{maxmin} problems are intractable}
		\State \Return \( X^* \) as defined in \eqref{Xstar}. \label{line9}
		\EndIf
		\EndFor
	\end{algorithmic}
\end{algorithm}
At this point, we emphasize a key distinction: the original MBQP problem~\eqref{bqp} involves minimizing a non-convex quadratic objective over the feasible set \( P \), whereas the subproblems in~\eqref{maxmin} involve linear objective functions over the same set \( P \), and thus constitute mixed-binary linear programs (MBLPs). Although both MBQP and MBLP problems are NP-hard and share the same feasible region, their practical tractability differs substantially. In practice, MILPs are often significantly easier to solve. For example, MILPs with over 200 variables can typically be solved within minutes using modern solvers, whereas MBQPs of comparable size may require hours. This disparity highlights the advantage of leveraging powerful MILP solvers to tackle problems that would otherwise be computationally prohibitive—an idea also explored in~\cite{bertsimas2016best,bertsimas2019machine} and \cite{anstreicher2021testing}.

The performance of the primal FRA is determined by the availability of feasible solutions. To clarify this, we broadly categorize the optimization problems in~\eqref{maxmin}  into the following classes:
\begin{sloppypar}
\begin{enumerate}
	\item \textbf{Polynomial-time solvable problems.} In some cases, the problems in~\eqref{maxmin} admit polynomial-time solutions. A canonical example is the classical minimum spanning tree problem. In contrast, its generalization—the quadratic minimum spanning tree problem (QMSTP)—is NP-hard and has been studied extensively; see~\cite{zhout1998effective,sundar2010swarm,cordone2012solving,pereira2015lower,oncan2010quadratic}. Recent works~\cite{de2024spanning,sotirov2024quadratic} explore SDP-based relaxations of the QMSTP, where our proposed facial reduction method can be directly applied. For such problems, the primal FRA is guaranteed to restore Slater’s condition in polynomial time.
	
	\item \textbf{NP-hard problems with tractable feasibility.} In many practical applications, the problems in~\eqref{maxmin} are NP-hard, yet finding feasible solutions is relatively easy. This makes our method broadly applicable, even if solving~\eqref{maxmin} to optimality is intractable in theory. In our experiments, we were able to restore Slater’s condition for a wide range of benchmark instances. When a new affinely independent feasible solution cannot be found, we terminate Algorithm~\ref{alg3} at Line~\ref{line9} and return the current feasible matrix. Although this matrix may not have maximal rank, it often results in a significant size reduction of the FR auxiliary problem and improved computational performance.
	
	\item \textbf{Intractable problems with difficult feasibility.} In rare cases, even finding a single feasible solution in \( P \) is computationally intractable. In such instances, our method cannot be applied to perform facial reduction. Fortunately, such cases are uncommon in practice, as evidenced by our empirical results on benchmark libraries such as MIPLIB.
\end{enumerate}
\end{sloppypar}

\noindent In summary, although the proposed approach may involve solving NP-hard subproblems, it requires only a single feasible solution to initiate facial reduction. In most practical settings, such a point can be found efficiently or generated heuristically. Consequently, the method is both effective and broadly applicable across a wide range of real-world optimization problems.

\begin{remark}
	In practice, some of the optimization problems in~\eqref{maxmin} may not be solvable to optimality within a reasonable timeframe. However, it is important to note that solving~\eqref{maxmin} to optimality is only necessary when verifying the equality condition in~\eqref{maineq}. In most cases, it suffices to find a feasible solution satisfying \( u_i^\top(x - v) \neq 0 \). This condition can often be efficiently checked during a branch-and-bound search. Specifically, whenever a new feasible solution is found, we verify whether \( u_i^\top(x - v) \neq 0 \) (up to a numerical tolerance). If so, we may terminate the search early and return the feasible solution.
\end{remark}

\begin{remark}
	\textbf{[Tightening and Regularization]} The Primal FRA extends the utility of facial reduction beyond standard regularization. While the standard FRA identifies the minimal face containing the relaxation set $L \cap \Snop$ (leaving the feasible set of the relaxation unchanged), the Primal FRA constructs the face $G$ using $X^*$ returned by Algorithm~\ref{alg3}, which is a convex combination of \emph{integer} feasible solutions. When Algorithm~\ref{alg3} terminates at Line~7, the range space of $X^{*}$ characterizes the smallest face $G$ of $\Snp$ containing the set $S$ defined in \eqref{liftedP}. This allows us to immediately conclude that $L \cap G$ satisfies Slater's condition without further computation. Furthermore, if $X^{*}$ does not have the maximum rank in $L \cap \Snop$, then $L \cap G \subsetneq L \cap \Snop$, meaning the reduced problem is a strictly stronger SDP relaxation than the original. Consequently, the Primal FRA simultaneously restores Slater's condition and \emph{tightens} the relaxation. We highlight this distinct advantage here, rather than formally incorporating it into Algorithm~\ref{alg2}, to keep the presentation of the main algorithmic framework simple to follow.
\end{remark}

\begin{remark}[Uniqueness vs. Heuristic Restriction]
	The properties of the facially reduced formulation depend heavily on the termination condition of Algorithm~\ref{alg3}. 
	
	\begin{itemize}
		\item \textbf{Exact Regime (Completion):} If the algorithm terminates at Line~7, the constructed matrix $X^*$ is guaranteed to have the maximum possible rank within the convex hull of the lifted feasible set $S$ in \eqref{liftedP}. In this case, the generated face $G$ is \textbf{unique}. The resulting reduction is independent of the specific sequence of points sampled. In addition, if the affine hull of $P$ has dimension $d$, we strictly require exactly $d+1$ affinely independent feasible points $x_{0},\ldots,x_{d}$ to recover the minimal face.
		
		\item \textbf{Heuristic Regime (Early Termination):} If the algorithm terminates early at Line~9 (e.g., due to computational limits), the matrix $X^*$ is merely a convex combination of the specific feasible points found up to that iteration. Consequently, the resulting face $G$ depends on the specific samples and may strictly contain the minimal face. In this regime, the reduced formulation is valid but \textbf{not unique}, representing a heuristic restriction based on the sampled subspace.
	\end{itemize}
\end{remark}

\subsection{Reducing the Number of Subproblems}

At the \( k \)-th iteration of Algorithm~\ref{alg3}, the orthogonal complement \( H^\perp \) is a subspace of dimension \( n - k + 1 \). Consequently, up to \( 2(n - k + 1) \) optimization problems of the form~\eqref{maxmin} may need to be solved. In the worst case, this results in approximately \( \mathcal{O}(n^2) \) subproblems across all iterations. In this section, we show that it is often possible to reduce this number significantly by exploiting the structure of the problem.

Suppose there exists a linear subspace \( \bar{H} \) of dimension \( n - \ell \) such that
\begin{equation}\label{Hbar}
	v + H \subseteq \operatorname{aff}(P) \subseteq v + \bar{H}
\end{equation}
for some feasible solution \( v \in P \) and current subspace \( H \). We first show that such a subspace \( \bar{H} \) can lead to a substantial reduction in the number of subproblems required. We then demonstrate how \( \bar{H} \) can often be identified directly from the linear constraints that define \( P \).

Since \( \operatorname{aff}(P) - v \subseteq \bar{H} \), it follows that $u^\top(x - v) = 0$ $\text{for all } x \in P \text{ and } u \in \bar{H}^\perp.$
As \( \bar{H}^\perp \subseteq H^\perp \), if we choose a set of linearly independent vectors \( u_1, \ldots, u_{n - k + 1} \) spanning \( H^\perp \), with the first \( \ell \) vectors \( u_1, \ldots, u_\ell \) spanning \( \bar{H}^\perp \), then we only need to solve~\eqref{maxmin} for \( u_{\ell+1}, \ldots, u_{n - k + 1} \). This reduces the number of subproblems in the \( k \)-th iteration by \( 2\ell \), which can be substantial when \( \ell > 0 \), i.e., when \( \bar{H} \) is not full-dimensional.

Furthermore, the maximum number of iterations in Algorithm~\ref{alg3} is reduced from \( n \) to \( n - \ell \). If the algorithm reaches the \( (n - \ell) \)-th iteration, the current subspace \( H \) has dimension \( \ell \), and since \( v + H \subseteq \operatorname{aff}(P) \subseteq v + \bar{H} \), and both subspaces share the same dimension, it follows that $\operatorname{aff}(P) = v + H = v + \bar{H}.$
Hence, no optimization problems need to be solved in the final iteration. This not only accelerates the algorithm but also provides a practical and effective termination criterion in our implementation. We describe such an application in Remark~\ref{Hbarapp}.

Importantly, the subspace \( \bar{H} \) in~\eqref{Hbar} often arises naturally from the definition of \( P \). Specifically, consider the linear programming relaxation of the mixed-binary set \( P \), given by:
\begin{equation}\label{bqpf_lp}
	\begin{array}{rlll}
		a_i^\top x & = & b_i & \text{for } i = 1, \ldots, p, \\
		a_i^\top x & \leq & b_i & \text{for } i = p+1, \ldots, m.
	\end{array}
\end{equation}
The affine hull of this polyhedron—obtained by identifying the set of (implicit) equality constraints—provides a natural candidate for the subspace \( \bar{H} \). This subspace can be extracted efficiently using standard techniques in linear programming, as discussed in \cite{freund1985identifying,schrijver1998theory,fukuda2016lecture,kelly2024applications}, which can significantly reduce the computational cost of Algorithm~\ref{alg3}.

\begin{remark}\label{Hbarapp}
	We discuss an application of the linear subspace \( \bar{H} \) in the implementation. Many MIPLIB problem instances are known to suffer from numerical instability due to scaling issues. To mitigate this, we restrict our attention to feasible solutions with bounded entries. Specifically, we define the restricted feasible set $\bar{P} := P \cap \{ x \in \mathbb{R}^n \mid x_i \in [-10^{5}, 10^{5}] \},$
	and search for feasible solutions within \( \bar{P} \) instead of the full set \( P \). As a consequence, we may no longer be able to certify that the affine hull of \( P \) has been fully recovered, since \( \operatorname{aff}(\bar{P}) \subseteq \operatorname{aff}(P) \). Nevertheless, if we succeed in identifying \( n - \ell + 1 \) affinely independent feasible solutions within \( \bar{P} \), we can still conclude that these solutions span the affine hull of \( P \).
	
	It is important to emphasize that facial reduction is still being performed on the original set \( P \); the restriction to \( \bar{P} \) is introduced solely to reduce numerical instability in practice. 
	
\end{remark}

\subsection{A Quadratic Speedup via Randomization}

While the technique discussed in the previous section reduces the number of subproblems, the total number of optimization problems remains \( \mathcal{O}(n^2) \) in the worst case. In this section, we demonstrate that it suffices to solve at most two optimization problems per iteration without imposing any additional structural assumptions. This yields a significant improvement, reducing the total number of required MILP problems to \( \mathcal{O}(n) \), thereby achieving a quadratic speedup.

\begin{lemma}\label{quad}
	Assume \( v + H \subsetneq \operatorname{aff}(P) \), and let \( u_1, \ldots, u_{n - k + 1} \in \mathbb{R}^n \) be linearly independent vectors spanning \( H^\perp \). Let \( \alpha_1, \ldots, \alpha_{n - k + 1} \in \mathbb{R} \) be i.i.d.\ random variables drawn from the standard normal distribution, and define \( u := \sum_{i=1}^{n - k + 1} \alpha_i u_i \). Consider the two problems
	\[
	\min \{ u^\top x \mid x \in P \} \quad \text{and} \quad \max \{ u^\top x \mid x \in P \}.
	\]
	Then, with probability $1$, the optimal value of at least one of these problems differs from \( u^\top v \) \rvv{(and equals \( -\infty \) or \( +\infty \) when the corresponding problem is unbounded). Consequently, there exists a feasible point \( x^* \in P \) with \( u^\top(x^* - v) \neq 0 \); when the relevant optimum is finite, \( x^* \) may be taken as the corresponding minimizer or maximizer.}
\end{lemma}

\begin{proof}
	Since \( v + H \subsetneq \operatorname{aff}(P) \), there exists some \( x \in P \) such that \( x - v \notin H \). Therefore, \( u_i^\top(x - v) \neq 0 \) for at least one \( i \). Because the coefficients \( \alpha_i \) are drawn independently from a continuous distribution, the linear combination \( u^\top(x - v) = \sum_{i=1}^{n - k + 1} \alpha_i u_i^\top(x - v) \neq 0 \) with probability 1. \rvv{Suppose \( u^\top(x - v) < 0 \); the case \( u^\top(x - v) > 0 \) is symmetric. Then the optimal value of \( \min\{ u^\top x \mid x \in P \} \) is at most \( u^\top x < u^\top v \), and hence differs from \( u^\top v \). If this optimum is finite, the minimizer \( x^*_{\min} \) satisfies \( u^\top(x^*_{\min} - v) \leq u^\top(x - v) < 0 \), so we may take \( x^* = x^*_{\min} \). If instead the problem is unbounded below, then any feasible point with objective value strictly less than \( u^\top v \), for instance \( x \) itself, satisfies \( u^\top(x^* - v) < 0 \). In either case, a witness point \( x^* \in P \) with \( u^\top(x^* - v) \neq 0 \) is obtained.}
\end{proof}

\rvv{The minimizer and maximizer above are guaranteed to exist whenever the linear objective \( u^\top x \) is bounded over \( P \), which holds in particular when \( P \) is bounded. In our implementation, we work with the bounded restricted set \( \bar{P} \) introduced in Remark~\ref{Hbarapp}, so both optima are always attained and the unbounded case does not arise in practice.}

%

\subsection{Numerical Robustness}
Numerical algorithms often struggle to distinguish between values that are truly zero and those that are merely close to zero due to rounding errors or the limitations of finite-precision arithmetic. In the context of FRAs, accurately identifying the range space of the feasible matrix is essential, as even minor misclassifications can compromise the correctness of the reduced problem.

In this section, we first clarify how numerical issues may arise in the primal FRA. (We emphasize, however, that such numerical difficulties were rarely encountered in our computational experiments.) We then propose a heuristic approach to mitigate these numerical issues when they do occur. Finally, we argue that in the rare event where numerical instability persists and the heuristic fails, one can eliminate these issues by taking a conservative inner approximation of the face \( G \). This simple modification guarantees numerical stability, at the cost of only a marginally smaller reduction in problem size.

We first clarify the potential numerical issues in the primal FRA. In Algorithm~\ref{alg2}, the face \( G \) is determined by the range space of the matrix \( X^{*} = VV^{\top} \), where \( V \) is constructed from feasible solutions \( x_0, \ldots, x_k \in P \) via Algorithm~\ref{alg3}:
\begin{equation}\label{Vdef}
	V := \frac{1}{\sqrt{k+1}} \begin{bmatrix}
		1 & \cdots & 1 \\
		x_0 & \cdots & x_k
	\end{bmatrix}.
\end{equation}
Assume the affine subspace spanned by $x_{0},\ldots,x_{k}$ is $\aff(P)$ and \( k < n \). Let \(\sum_{i=1}^{k+1} \sigma_i u_i v_i^{\top} \) be the numerical singular value decomposition of $V$. In practice, the range space of \( X^* \) is numerically approximated as \( \operatorname{span}\{ u_i \mid \sigma_i > \delta \} \), for some threshold \( \delta \). However, ambiguity arises when \( \sigma_i \in (\epsilon_{\text{mach}}, \delta] \), where \( \epsilon_{\text{mach}} \) denotes the machine epsilon—the smallest number such that \( 1 + \epsilon_{\text{mach}} \neq 1 \) in floating-point arithmetic (typically around \( 2.22 \times 10^{-16} \) for double-precision). Singular values in this range may be inconsistently treated as either zero or nonzero by numerical solvers. Misinterpreting a near-zero singular value as positive can lead to overestimating the range space, potentially yielding an incorrect reduced face \( L^{\perp} \cap G^{\triangle} \subsetneq L^{\perp} \cap \mathbb{S}_+^n \).

To mitigate this issue, we propose augmenting the matrix \( V \) by including an additional feasible solution \( x_{k+1} \in P \) that increases an ambiguous singular value \( \sigma_r \). Let \( u_r = \begin{bmatrix} h_r \\ g_r \end{bmatrix} \) be the left singular vector of \( V \) corresponding to \( \sigma_r \). We solve:
\begin{equation}\label{mineig}
	\max \left\{ g_r^{\top} x \mid x \in P \right\},
\end{equation}
to find a new point \( x_{k+1} \) that aligns with direction \( g_r \), thus improving the conditioning of the extended matrix:
\[
\bar{V} := \frac{1}{\sqrt{k+2}} \begin{bmatrix}
	1 & \cdots & 1 & 1 \\
	x_0 & \cdots & x_k & x_{k+1}
\end{bmatrix}.
\]
The corresponding matrix \( \bar{X} = \bar{V} \bar{V}^{\top} \) remains feasible for $\conv(S)$, and results in the same face \( G \), but with potentially improved numerical stability.

We formalize this improvement below:

\begin{lemma}\label{improvecond}
\rvv{Let \( V \) be defined as in \eqref{Vdef}, and suppose \( \sigma_r(V) < 1 \) is the \( r \)-th singular value. Let \( u_r = \begin{bmatrix} h_r \\ g_r \end{bmatrix} \) be the corresponding left singular vector. Define:
$
	U := \frac{1}{\sqrt{k+2}} \begin{bmatrix}
		1 & \cdots & 1 & h_r \\
		x_0 & \cdots & x_k & g_r
	\end{bmatrix}.
$
	Then \( u_r \) is an eigenvector of \( UU^{\top} \) whose associated eigenvalue is strictly greater than \( \sigma_r^2(V) \).}
\end{lemma}

\begin{proof}
	\rvv{Let \( X^* = VV^{\top} \) and \( Y^* = UU^{\top} \). By the definition of the singular value decomposition, \( u_r \) is an eigenvector of \( X^* \) corresponding to the eigenvalue \( \lambda_r(X^*) = \sigma_r^2(V) \). From the construction of \( U \), the matrix \( Y^* \) can be expressed as a rank-1 update of \( X^* \):
	\[
	Y^* = \frac{k+1}{k+2} X^* + \frac{1}{k+2} u_r u_r^{\top}.
	\]
	Since \( u_r \) is a unit vector and an eigenvector of \( X^* \), we examine the action of \( Y^* \) on \( u_r \):
	\[
	\begin{aligned}
		Y^* u_r &= \left( \frac{k+1}{k+2} X^* + \frac{1}{k+2} u_r u_r^{\top} \right) u_r \\
		&= \frac{k+1}{k+2} (X^* u_r) + \frac{1}{k+2} u_r (u_r^{\top} u_r) \\
		&= \frac{k+1}{k+2} \sigma_r^2(V) u_r + \frac{1}{k+2} u_r \\
		&= \left( \frac{(k+1)\sigma_r^2(V) + 1}{k+2} \right) u_r.
	\end{aligned}
	\]
	This demonstrates that \( u_r \) remains an eigenvector of \( Y^* \) with the associated eigenvalue
	\[
	\frac{(k+1)\sigma_r^2(V) + 1}{k+2}.
	\]
	Since \( \sigma_r(V) < 1 \), we have:
	\[
	\frac{(k+1)\sigma_r^2(V) + 1}{k+2} > \sigma_r^2(V) \iff (k+1)\sigma_r^2(V) + 1 > (k+2)\sigma_r^2(V) \iff 1 > \sigma_r^2(V),
	\]
	which holds by assumption. Thus, the eigenvalue of \( Y^* = UU^{\top} \) in the direction \( u_r \) is strictly greater than \( \sigma_r^2(V) \).}
\end{proof}

Of course, the vector \( g_r \) in Lemma~\ref{improvecond} may not lie in \( P \). In practice, we approximate this direction by using a feasible solution \( x \) that maximizes \( g_r^T x \). This heuristic motivates the formulation in \eqref{mineig}. Although we do not have a formal bound for \eqref{mineig}, our experiments consistently show that this approach improves the conditioning of \( V \)—increasing the target singular value and reducing ambiguity in rank determination.

In rare cases where ambiguous singular values persist—indicating either the failure of the above strategy or an ill-conditioned feasible set \( P \)—we simply treat all such singular values as zero. This leads to the following inner approximation: $\bar{G} := \operatorname{span} \left\{ u_i \mid \sigma_i(V) > \bar{\delta} \right\},$
where $\bar{\delta}$ is chosen conservatively, e.g., \( \bar{\delta} \in [10^{-5}, 10^{-3}] \).  Since \( \bar{G} \subseteq G \), it follows that \( \bar{G}^\triangle \supseteq G^\triangle \), and therefore: $
L^\perp \cap G^\triangle \subseteq L^\perp \cap \bar{G}^\triangle.$
With a sufficiently large \( \bar{\delta} \), we can confidently conclude: $
L^\perp \cap \bar{G}^\triangle = L^\perp \cap \mathbb{S}_+^n.
$

In practice, the number of ambiguous singular values is typically zero or very small, and the proposed inner approximation \( \bar{G} \) provides a robust and reliable implementation of the primal FRA, while preserving nearly all of its reduction capability.
\begin{remark}
	While \( \bar{\delta} \) can improve numerical robustness in some cases, we did not enable this feature in our implementation, as we applied the same settings uniformly across all problem instances without instance-specific tuning. 
\end{remark}
\begin{sloppypar}
\begin{remark}
		Note that we only claim the primal FRA does not introduce additional numerical issues. If an SDP problem is ill-posed and numerical difficulties arise in the standard FRA—specifically, in Line~3 of Algorithm~\ref{alg1}—then similar numerical issues are likely to occur when determining the face \( F \) in the primal FRA formulation~\eqref{newF}. 
		
		The key advantage of the primal FRA lies in its ability to compute the exposing vector more efficiently, without introducing new sources of numerical instability. However, it does not eliminate existing numerical challenges inherent to the original problem.
\end{remark}
\end{sloppypar}

\section{Relation to Other Approaches}

In this section, we discuss the relationship between the primal FRA and several alternative FRAs. The primal FRA is highly compatible with many of these alternatives, allowing us to combine their advantages rather than choosing a single method. We illustrate this synergy using partial FRA \cite{permenter2018partial} as an example. We begin with a brief overview of the key ideas behind partial FRA and then describe how it integrates with the primal FRA. We also compare the primal FRA with the affine FRA proposed in \cite{hu2023affine}, highlighting important distinctions and advantages.

Partial FRA replaces the positive semidefinite cone \(\mathbb{S}_+^n\) with an inner approximation \(K \subseteq \mathbb{S}_+^n\) such that \(L^{\perp} \cap K \subseteq L^{\perp} \cap \mathbb{S}_+^n\). By choosing \(K\) to be computationally simpler, one obtains a more practical method for identifying an exposing vector. A widely used choice is the cone of diagonally dominant matrices, defined by: $
\mathcal{DD}^n := \left\{ X \in \mathbb{S}^n \;\middle|\; X_{ii} \geq \sum_{j \neq i} |X_{ij}| \right\}.$
Since \(\mathcal{DD}^n \subseteq \mathbb{S}_+^n\), searching for a nonzero matrix in \(L^{\perp} \cap \mathcal{DD}^n\) reduces to solving a linear program. The trade-off is that \(L^{\perp} \cap \mathcal{DD}^n\) may contain only the zero matrix, even when \(L^{\perp} \cap \mathbb{S}_+^n\) contains a nonzero exposing vector—resulting in failure to detect a valid reduction.

In contrast, the primal FRA identifies a proper face \(G^{\triangle}\) such that \(L^{\perp} \cap \mathbb{S}_+^n = L^{\perp} \cap G^{\triangle}\), thereby reducing the dimension of the FR auxiliary problem. These two methods can be effectively combined. One may first apply the primal FRA to obtain a reformulation over a smaller face \(G^{\triangle}\), and then apply partial FRA to construct an inner approximation \(\tilde{K} \subseteq G^{\triangle}\). Since \(G^{\triangle}\) is a proper face of \(\mathbb{S}_+^n\), the resulting \(\tilde{K}\) is lower-dimensional and more manageable than an inner approximation over the full cone. For example, if
$
G^{\triangle} = \left\{ [ \begin{smallmatrix}
	R & 0 \\
	0 & 0
\end{smallmatrix} ] \;\middle|\; R \in \mathbb{S}_+^r \right\}, r < n$,
then a natural inner approximation is
$
\tilde{K} = \left\{  [ \begin{smallmatrix}
	R & 0 \\
	0 & 0
\end{smallmatrix} ]  \;\middle|\; R \in \mathcal{DD}^r \right\}.
$
This demonstrates how primal FRA and partial FRA can be seamlessly combined, benefiting from both reduction quality and computational efficiency. The same strategy applies to other methods such as Sieve-SDP \cite{zhu2019sieve} and the preprocessing approach in \cite{cheung2013preprocessing}.

The affine FRA proposed in \cite{hu2023affine} adopts a different perspective: it identifies implicit equalities by analyzing the affine hull of the feasible region \(P\) through its LP relaxation. This yields an affine subspace containing \(P\), although not necessarily the minimal one. The affine FRA performs well in practice, particularly on benchmark instances from MIPLIB. It often achieves the same reduction as the primal FRA, sometimes with less computational effort. A notable limitation of the affine FRA is its inability to verify whether Slater's condition holds for the reduced problem. Without this verification, the reliability of the obtained numerical optimal solution remains uncertain. In contrast, while the primal FRA does not universally guarantee strict feasibility, it explicitly attempts to verify Slater's condition. As demonstrated in our numerical experiments, the primal FRA successfully confirms that Slater's condition holds for most instances. This capability provides a critical advantage in applications where numerical precision and stability are paramount.

\section{Numerical Experiments}\label{sec_num}
\subsection{Settings}
\textbf{Experimental Settings:} All experiments were conducted on a Mac Studio (2023) equipped with an Apple M2 Ultra chip, 128 GB of RAM, and running macOS 14.1.1 (build 23B81). We generated MBQP problem instances based on benchmark instances from MIPLIB 2017~\cite{MIPLIB2017}, specifically those categorized under ``collection.'' The instances were imported into MATLAB R2023b using Gurobi (version 10.0.1)~\cite{gurobi}.  We also used Gurobi to solve problem~\eqref{maxmin}, which generates feasible solutions within Algorithm~\ref{alg3}. The SDP problems were solved using the MOSEK solver in~\cite{aps2019mosek} via the MATLAB interface YALMIP~\cite{lofberg2004yalmip}.

\textbf{Problem Generation:} The MIPLIB provides MILP problems whose feasible regions are of the same form as in~\eqref{bqpf}. In addition, MIPLIB specifies a linear objective function \( c^{T}x \) for some vector \( c \in \mathbb{R}^n \). To construct MBQP instances, we generate a random symmetric quadratic cost matrix \( Q \in \mathbb{S}^n \), where the entries are independently drawn from a standard normal distribution. This results in an MBQP problem of the form~\eqref{bqp}. Note that the objective function does not affect facial reduction; it is used only when solving the original SDP problem in the Appendix.

\textbf{SDP Relaxations:} For small to moderate problem instances, we apply the standard FRA and the primal FRA to the following three SDP relaxations:

 \textbf{1. Shor's Relaxation.}	Define the matrices: $\tilde{Q}:= [\begin{smallmatrix}
		0 & \frac{1}{2}c^{\top} \\
		\frac{1}{2}c & Q
	\end{smallmatrix}] \in \mathbb{S}^{n+1} \text{ and } A_i := [\begin{smallmatrix}
		0 & \frac{1}{2}a_i^{\top} \\
		\frac{1}{2}a_i & \mathbf{O}
\end{smallmatrix}] \in \mathbb{S}^{n+1}, \quad i = 1, \ldots, m.$
	\rvv{This coincides with the lifted objective matrix in \eqref{sdprelaxopt}. Since facial reduction depends only on the feasible set, the preprocessing experiments are unchanged by the choice of objective; we include the full \( \tilde{Q} \) for consistency with the general SDP framework and the supplementary solves in Appendix~\ref{sdpnum}.}
	Recall that $\arrow$ is the linear operator defined in \eqref{def_arrow}. Then the relaxation is given by:
	\begin{align}
		\min \quad & \langle \tilde{Q}, Y \rangle \notag \\
		\text{s.t.} \quad
		& \langle A_i, Y \rangle = b_i, \quad i = 1, \ldots, p \label{con:ineq} \\
		& \langle A_i, Y \rangle \leq b_i, \quad i = p+1, \ldots, m \label{con:eq} \\
		& \arrow(Y) = e_0 \label{con:arrow} \\
		& Y \in \mathbb{S}_+^{n+1}. \label{con:psd}
	\end{align}

\textbf{2. Doubly Nonnegative (DNN) Relaxation.} Define $\bar{A}_i := [\begin{smallmatrix}
		-b_i \\
		a_i
	\end{smallmatrix}]
	[\begin{smallmatrix}
		-b_i & a_i^{\top}
	\end{smallmatrix}] \in \mathbb{S}^{n+1}.$
	In many problems, the inequality constraints include variable bounds \( l \leq x \leq u \). We can derive  additional constraints based on the Reformulation-Linearization Technique (RLT) \cite{sherali1990hierarchy,sherali1994hierarchy}. This yields the following SDP relaxation:
	\begin{equation*}
		\begin{array}{rll}
			\min \quad & \langle \tilde{Q}, Y \rangle \\
			\text{s.t.} \quad
			& Y \text{ satisfies } \eqref{con:ineq}, \eqref{con:eq}, \eqref{con:arrow}, \eqref{con:psd} \\
			& \langle \bar{A}_i, Y \rangle = 0, \quad i = 1, \ldots, p \\
			& l_i l_j - Y_{0i}l_j - Y_{0j}l_i + Y_{ij} \geq 0, \quad \forall i,j \\
			& u_i u_j - Y_{0i}u_j - Y_{0j}u_i + Y_{ij} \geq 0, \quad \forall i,j.
		\end{array}
	\end{equation*}
	When the problem contains only binary variables, the matrix \( Y \) is both positive semidefinite and entrywise nonnegative; hence, we refer to this as the \emph{doubly nonnegative (DNN)} relaxation.
	
\textbf{3. A Variant of Shor's Relaxation.}  For large-scale instances, solving Shor's or DNN relaxations may be computationally prohibitive. We consider a special case in which quadratic costs appear only between binary variables, i.e., $Q = [\begin{smallmatrix}
		Q_{B} & \textbf{O} \\
		\textbf{O} & \textbf{O} \\
	\end{smallmatrix}] \text{ for some } Q_B \in \mathbb{S}^r.$
	Here, \( Q_B \in \mathbb{S}^r \) represents the quadratic cost matrix for binary variables.  Let the first \( r \) variables be binary, and decompose: $x = [\begin{smallmatrix}
		\tilde{x} \\
		\hat{x}
	\end{smallmatrix}]$, where $\tilde{x} \in \mathbb{R}^r$ and $\hat{x} \in \mathbb{R}^{n - r}$.
	Then the MBQP  in \eqref{bqp} can be written as:
	\begin{equation}\label{bqp_bq}
		\min \left\{ \tilde{x}^{\top} Q_B \tilde{x} + c^{\top}x \,\middle|\, x \in P \right\}.
	\end{equation}
	Let \( \tilde{a}_i \in \mathbb{R}^r \), \( \hat{a}_i \in \mathbb{R}^{n - r} \), \( \tilde{c} \in \mathbb{R}^r \) and \( \hat{c} \in \mathbb{R}^{n - r} \) be defined analogously. Define: $\tilde{Q}:= [\begin{smallmatrix}
		0 & \tilde{c}^{T}\\
		\tilde{c} & Q_{B}
	\end{smallmatrix}] \in \mathbb{S}^{r+1} \text{ and } 
	\tilde{A}_i := [\begin{smallmatrix}
		0 & \frac{1}{2} \tilde{a}_i^{\top} \\
		\frac{1}{2} \tilde{a}_i & \mathbf{O}
	\end{smallmatrix}] \in \mathbb{S}^{r+1}.$
	The following is an SDP relaxation for \eqref{bqp_bq}:
	\begin{equation}
		\begin{array}{rll}
			\min \quad & \langle \tilde{Q}, X \rangle + \hat{c}^{\top} \hat{x} \\
			\text{s.t.} \quad
			& \langle \tilde{A}_i, X \rangle + \hat{a}_i^{\top} \hat{x} = b_i, \quad i = 1, \ldots, p \\
			& \langle \tilde{A}_i, X \rangle + \hat{a}_i^{\top} \hat{x} \leq b_i, \quad i = p+1, \ldots, m \\
			& \arrow(X) = e_0 \\
			& X \in \mathbb{S}_+^{r+1}.
		\end{array}
	\end{equation}
	Here, the matrix variable \( X \) is only of order \( r+1 \), which is typically much smaller than \( n+1 \), leading to a significant reduction in computational cost. This significantly reduces computational cost and enables the testing of larger problem instances. We will call \eqref{bqp_bq} a \emph{variant of Shor's relaxation}. To the best of our knowledge, the variant of Shor's relaxation is new in the literature.

\textbf{Facial Reduction:} In our computational result, we only apply the first iteration of the FRA. There are three main reasons for this choice. First, the backward stability of a single iteration of facial reduction has been established in \cite{cheung2013preprocessing}. This implies that the computed solution to the reduced problem is the exact solution to a problem that is arbitrarily close to the original one, ensuring numerical reliability. In contrast, it remains an open question whether the same guarantee holds for subsequent iterations. Second, as a preprocessing algorithm, FRA should adhere to the principle of being \emph{simple and quick}, as advocated by \cite{andersen1995presolving}. Performing a second iteration of facial reduction is often too computationally expensive to be practical in this context. Third, the primal FRA is highly effective in practice, often restoring Slater's condition immediately after just one iteration for the majority of benchmark instances. \rvv{In the rare cases where a single primal FRA iteration reduces the problem but does not fully restore Slater's condition, one may continue with a subsequent iteration. Since the primal sampling heuristic does not generate additional affinely independent feasible points beyond those already found, one must instead switch to a standard FRA step (Algorithm~\ref{alg1}) applied to the already reduced face $G$, rather than perform another primal iteration. This auxiliary problem is strictly smaller than the original, and by the finite convergence of FRA \cite{borwein1981facial,liu2018exact,pataki2013strong,waki2013facial}, alternating with standard FRA iterations is guaranteed to restore Slater's condition in finitely many steps. We note, however, that severe ill-conditioning specific to the geometry of the reduced face may still persist.}

\subsection{Results and Discussion}

We report the numerical performance of the standard FRA and the proposed primal FRA on a diverse set of benchmark instances. We begin by presenting and discussing overall performance statistics to highlight the key advantages of the primal FRA. Detailed numerical results for individual problem instances follow, providing a deeper analysis of algorithmic behavior across different SDP relaxations.

Table~\ref{tab:summary} provides an aggregated comparison between the standard FRA and the proposed primal FRA across three types of SDP relaxations and a total of 157 \rvv{instance--relaxation pairs (arising from 91 distinct MIPLIB instances, each tested under one or more of the three relaxation types)}. The following key performance metrics are reported:

\begin{enumerate}
	\item \textbf{Total Time}: The total computation time in seconds required to perform FRA across all instances. (If the standard FR auxiliary problem could not be solved due to memory limitations, we exclude the computational time for both FRA methods.) The primal FRA achieves a substantial reduction in total computation time—from $31,\!876$ seconds to $6,\!975$ seconds—yielding a fourfold speedup.
	
	\item \textbf{Number of Successful Solves}: The number of instances for which the FR auxiliary problem was successfully solved. The primal FRA achieves a higher success rate ($146$ out of $157$) compared to the standard FRA ($120$ out of $157$), demonstrating its improved robustness.
	
	\item \textbf{Slater Condition Restored}: The number of instances in which Slater’s condition was successfully restored. For the standard FRA, we apply one step of Algorithm~\ref{alg1}. The primal FRA restores Slater’s condition in $145$ instances, while the standard FRA does so in only $108$, further demonstrating the effectiveness of the primal approach.
	
	\item \textbf{Average Size Ratio}: For each instance, we compute the ratio between the matrix variable size in the reduced FR auxiliary problem from the primal FRA and that from the standard FRA. The reported value is the average ratio across all instances. On average, the matrix variable in the primal FRA is approximately $9\%$ the size of that in the standard FRA, indicating a substantial reduction in problem size.
\end{enumerate}

These results show that the primal FRA not only improves computational efficiency but also enhances solver reliability and produces significantly more compact SDP formulations. This makes the primal FRA a promising alternative to the standard FRA for solving large-scale semidefinite relaxations. Detailed numerical results are provided in Appendix~\ref{num_details}.

\begin{table}[H]
	\centering
	\small
	\begin{tabular}{c|cc} \hline
		\textbf{Metric} & \textbf{Standard FRA} & \textbf{Primal FRA} \\ \hline
		Total Time                                   &    $31876$   &    $6975$    \\ 
		Successful Solves (Count)                    &    $120$     &    $146$     \\ 
		Slater Condition Restored (Count)            &    $108$     &    $145$     \\ 
		Average Size Ratio                           &    $1$       &    $0.09179$ \\ 
		\hline 
	\end{tabular}
	\caption{Summary of total time, number of successful solves, and Slater condition restorations over 157 instance--relaxation pairs\rvv{, arising from 91 distinct MIPLIB instances tested across three relaxation types}.}
	\label{tab:summary}
\end{table}

\section{Conclusion}

We have proposed a novel facial reduction algorithm (FRA) tailored to semidefinite relaxations of combinatorial optimization problems. Unlike traditional approaches that require solving computationally intensive FR auxiliary problems, the primal FRA formulates a reduced FR auxiliary problem that is significantly more tractable and efficient.

The proposed method integrates seamlessly with other specialized FRAs, such as partial FRA and Sieve-SDP, enabling complementary preprocessing strategies. Moreover, it exhibits strong numerical robustness, successfully addressing instances where the standard FRA fails due to numerical instability or memory limitations.

Extensive computational experiments demonstrate that the primal FRA substantially reduces the size and complexity of the FR auxiliary problem, restores Slater's condition in the majority of benchmark instances, and improves solver reliability—all while incurring minimal preprocessing overhead. These results underscore the practical value of the primal FRA as an effective and scalable tool for enhancing the performance of SDP relaxations in combinatorial optimization.

Future research could explore the application of Primal FRA within a branch and bound framework. In such algorithms, child nodes are generated by fixing specific variables to zero or one, creating subproblems whose feasible regions are closely connected to that of the root node. Investigating how to leverage this structural connection to efficiently apply Primal FRA to child nodes could allow for the efficient regularization of SDP relaxations throughout the entire search tree.

\bibliographystyle{siam}
\bibliography{literature}

\begin{thebibliography}{10}

\bibitem{andersen1995presolving}
{\sc E.~D. Andersen and K.~D. Andersen}, {\em Presolving in linear
  programming}, Mathematical Programming, 71 (1995), pp.~221--245.

\bibitem{anjos2011handbook}
{\sc M.~F. Anjos and J.~B. Lasserre}, {\em Handbook on semidefinite, conic and
  polynomial optimization}, vol.~166, Springer Science \& Business Media, 2011.

\bibitem{anstreicher2021testing}
{\sc K.~M. Anstreicher}, {\em Testing copositivity via mixed--integer linear
  programming}, Linear Algebra and Its Applications, 609 (2021), pp.~218--230.

\bibitem{aps2019mosek}
{\sc M.~ApS}, {\em Mosek optimization toolbox for {MATLAB}}, User's Guide and
  Reference Manual, version, 4 (2019).

\bibitem{araujo2023quantum}
{\sc M.~Ara{\'u}jo, M.~Huber, M.~Navascu{\'e}s, M.~Pivoluska, and A.~Tavakoli},
  {\em Quantum key distribution rates from semidefinite programming}, Quantum,
  7 (2023), p.~1019.

\bibitem{bertsimas2019machine}
{\sc D.~Bertsimas and J.~Dunn}, {\em Machine learning under a modern
  optimization lens}, Dynamic Ideas LLC Waltham, 2019.

\bibitem{bertsimas2016best}
{\sc D.~Bertsimas, A.~King, and R.~Mazumder}, {\em Best subset selection via a
  modern optimization lens}, The Annals of Statistics, 44 (2016), pp.~813--852.

\bibitem{biswas2006semidefinite}
{\sc P.~Biswas, T.-C. Liang, K.-C. Toh, Y.~Ye, and T.-C. Wang}, {\em
  Semidefinite programming approaches for sensor network localization with
  noisy distance measurements}, IEEE transactions on automation science and
  engineering, 3 (2006), pp.~360--371.

\bibitem{borwein1981regularizing}
{\sc J.~Borwein and H.~Wolkowicz}, {\em Regularizing the abstract convex
  program}, Journal of Mathematical Analysis and Applications, 83 (1981),
  pp.~495--530.

\bibitem{borwein1981facial}
{\sc J.~M. Borwein and H.~Wolkowicz}, {\em Facial reduction for a cone-convex
  programming problem}, Journal of the Australian Mathematical Society, 30
  (1981), pp.~369--380.

\bibitem{burkard1984quadratic}
{\sc R.~E. Burkard}, {\em Quadratic assignment problems}, European Journal of
  Operational Research, 15 (1984), pp.~283--289.

\bibitem{cheung2013preprocessing}
{\sc Y.-L. Cheung, S.~Schurr, and H.~Wolkowicz}, {\em Preprocessing and
  regularization for degenerate semidefinite programs}, in Computational and
  Analytical Mathematics: In Honor of Jonathan Borwein's 60th Birthday,
  Springer, 2013, pp.~251--303.

\bibitem{cordone2012solving}
{\sc R.~Cordone and G.~Passeri}, {\em Solving the quadratic minimum spanning
  tree problem}, Applied Mathematics and Computation, 218 (2012),
  pp.~11597--11612.

\bibitem{de2011lasserre}
{\sc E.~De~Klerk and M.~Laurent}, {\em On the lasserre hierarchy of
  semidefinite programming relaxations of convex polynomial optimization
  problems}, SIAM Journal on Optimization, 21 (2011), pp.~824--832.

\bibitem{de2024spanning}
{\sc F.~de~Meijer, M.~Siebenhofer, R.~Sotirov, and A.~Wiegele}, {\em Spanning
  and splitting: Integer semidefinite programming for the quadratic minimum
  spanning tree problem}, arXiv preprint arXiv:2410.04997,  (2024).

\bibitem{fawzi2018efficient}
{\sc H.~Fawzi and O.~Fawzi}, {\em Efficient optimization of the quantum
  relative entropy}, Journal of Physics A: Mathematical and Theoretical, 51
  (2018), p.~154003.

\bibitem{fawzi2023optimal}
{\sc H.~Fawzi and J.~Saunderson}, {\em Optimal self-concordant barriers for
  quantum relative entropies}, SIAM Journal on Optimization, 33 (2023),
  pp.~2858--2884.

\bibitem{finke1987quadratic}
{\sc G.~Finke, R.~E. Burkard, and F.~Rendl}, {\em Quadratic assignment
  problems}, in North-Holland Mathematics Studies, vol.~132, Elsevier, 1987,
  pp.~61--82.

\bibitem{freund1985identifying}
{\sc R.~M. Freund, R.~Roundy, and M.~J. Todd}, {\em Identifying the set of
  always-active constraints in a system of linear inequalities by a single
  linear program},  (1985).

\bibitem{friberg2016facial}
{\sc H.~A. Friberg}, {\em Facial reduction heuristics and the motivational
  example of mixed-integer conic optimization}, Optimization online, 5 (2016),
  p.~14.

\bibitem{fukuda2016lecture}
{\sc K.~Fukuda}, {\em Lecture: Polyhedral computation, spring 2016}, Institute
  for Operations Research and Institute of Theoretical Computer Science, ETH
  Zurich. https://inf. ethz.
  ch/personal/fukudak/lect/pclect/notes2015/PolyComp2015. pdf,  (2016).

\bibitem{MIPLIB2017}
{\sc A.~Gleixner, G.~Hendel, G.~Gamrath, T.~Achterberg, M.~Bastubbe,
  T.~Berthold, P.~M. Christophel, K.~Jarck, T.~Koch, J.~Linderoth,
  M.~L\"ubbecke, H.~D. Mittelmann, D.~Ozyurt, T.~K. Ralphs, D.~Salvagnin, and
  Y.~Shinano}, {\em {MIPLIB 2017: Data-Driven Compilation of the 6th
  Mixed-Integer Programming Library}}, Mathematical Programming Computation,
  (2021).

\bibitem{goemans1995improved}
{\sc M.~X. Goemans and D.~P. Williamson}, {\em Improved approximation
  algorithms for maximum cut and satisfiability problems using semidefinite
  programming}, Journal of the ACM (JACM), 42 (1995), pp.~1115--1145.

\bibitem{grotschel2012geometric}
{\sc M.~Gr{\"o}tschel, L.~Lov{\'a}sz, and A.~Schrijver}, {\em Geometric
  algorithms and combinatorial optimization}, vol.~2, Springer Science \&
  Business Media, 2012.

\bibitem{gurobi}
{\sc {Gurobi Optimization, LLC}}, {\em {Gurobi Optimizer Reference Manual}},
  2023.

\bibitem{hu2022robust}
{\sc H.~Hu, J.~Im, J.~Lin, N.~L{\"u}tkenhaus, and H.~Wolkowicz}, {\em Robust
  interior point method for quantum key distribution rate computation},
  Quantum, 6 (2022), p.~792.

\bibitem{hu2023affine}
{\sc H.~Hu and B.~Yang}, {\em Affine {FR}: an effective facial reduction
  algorithm for semidefinite relaxations of combinatorial problems}.
\newblock 2023.

\bibitem{hu2016note}
{\sc J.~Hu, B.~Jiang, X.~Liu, and Z.~Wen}, {\em A note on semidefinite
  programming relaxations for polynomial optimization over a single sphere},
  Science China Mathematics, 59 (2016), pp.~1543--1560.

\bibitem{jibetean2005semidefinite}
{\sc D.~Jibetean and M.~Laurent}, {\em Semidefinite approximations for global
  unconstrained polynomial optimization}, SIAM Journal on Optimization, 16
  (2005), pp.~490--514.

\bibitem{karimi2025efficient}
{\sc M.~Karimi and L.~Tuncel}, {\em Efficient implementation of interior-point
  methods for quantum relative entropy}, INFORMS Journal on Computing, 37
  (2025), pp.~3--21.

\bibitem{kelly2024applications}
{\sc S.~Kelly}, {\em Applications of Conic Programming Reformulations}, PhD
  thesis, Clemson University, 2024.

\bibitem{kim2009exploiting}
{\sc S.~Kim, M.~Kojima, and H.~Waki}, {\em Exploiting sparsity in sdp
  relaxation for sensor network localization}, SIAM Journal on Optimization, 20
  (2009), pp.~192--215.

\bibitem{krislock2010explicit}
{\sc N.~Krislock and H.~Wolkowicz}, {\em Explicit sensor network localization
  using semidefinite representations and facial reductions}, SIAM Journal on
  Optimization, 20 (2010), pp.~2679--2708.

\bibitem{liu2018exact}
{\sc M.~Liu and G.~Pataki}, {\em Exact duals and short certificates of
  infeasibility and weak infeasibility in conic linear programming},
  Mathematical Programming, 167 (2018), pp.~435--480.

\bibitem{lofberg2004yalmip}
{\sc J.~Lofberg}, {\em Yalmip: A toolbox for modeling and optimization in
  matlab}, in 2004 IEEE international conference on robotics and automation
  (IEEE Cat. No. 04CH37508), IEEE, 2004, pp.~284--289.

\bibitem{nesterov1994interior}
{\sc Y.~Nesterov and A.~Nemirovskii}, {\em Interior-point polynomial algorithms
  in convex programming}, SIAM, 1994.

\bibitem{oncan2010quadratic}
{\sc T.~{\"O}ncan and A.~P. Punnen}, {\em The quadratic minimum spanning tree
  problem: A lower bounding procedure and an efficient search algorithm},
  Computers \& Operations Research, 37 (2010), pp.~1762--1773.

\bibitem{papp2017semi}
{\sc D.~Papp}, {\em Semi-infinite programming using high-degree polynomial
  interpolants and semidefinite programming}, SIAM Journal on Optimization, 27
  (2017), pp.~1858--1879.

\bibitem{pataki2013strong}
{\sc G.~Pataki}, {\em Strong duality in conic linear programming: facial
  reduction and extended duals}, in Computational and Analytical Mathematics:
  In Honor of Jonathan Borwein's 60th Birthday, Springer, 2013, pp.~613--634.

\bibitem{pereira2015lower}
{\sc D.~L. Pereira, M.~Gendreau, and A.~S. Da~Cunha}, {\em Lower bounds and
  exact algorithms for the quadratic minimum spanning tree problem}, Computers
  \& Operations Research, 63 (2015), pp.~149--160.

\bibitem{permenter2018partial}
{\sc F.~Permenter and P.~Parrilo}, {\em Partial facial reduction: simplified,
  equivalent sdps via approximations of the psd cone}, Mathematical
  Programming, 171 (2018), pp.~1--54.

\bibitem{schrijver1998theory}
{\sc A.~Schrijver}, {\em Theory of linear and integer programming}, John Wiley
  \& Sons, 1998.

\bibitem{sherali1990hierarchy}
{\sc H.~D. Sherali and W.~P. Adams}, {\em A hierarchy of relaxations between
  the continuous and convex hull representations for zero-one programming
  problems}, SIAM Journal on Discrete Mathematics, 3 (1990), pp.~411--430.

\bibitem{sherali1994hierarchy}
\leavevmode\vrule height 2pt depth -1.6pt width 23pt, {\em A hierarchy of
  relaxations and convex hull characterizations for mixed-integer zero—one
  programming problems}, Discrete Applied Mathematics, 52 (1994), pp.~83--106.

\bibitem{sotirov2024quadratic}
{\sc R.~Sotirov and Z.~Verch{\'e}re}, {\em The quadratic minimum spanning tree
  problem: lower bounds via extended formulations}, Vietnam Journal of
  Mathematics,  (2024), pp.~1--22.

\bibitem{sundar2010swarm}
{\sc S.~Sundar and A.~Singh}, {\em A swarm intelligence approach to the
  quadratic minimum spanning tree problem}, Information Sciences, 180 (2010),
  pp.~3182--3191.

\bibitem{tavakoli2024semidefinite}
{\sc A.~Tavakoli, A.~Pozas-Kerstjens, P.~Brown, and M.~Ara{\'u}jo}, {\em
  Semidefinite programming relaxations for quantum correlations}, Reviews of
  Modern Physics, 96 (2024), p.~045006.

\bibitem{tunccel2001slater}
{\sc L.~Tun{\c{c}}el}, {\em On the slater condition for the sdp relaxations of
  nonconvex sets}, Operations Research Letters, 29 (2001), pp.~181--186.

\bibitem{waki2006sums}
{\sc H.~Waki, S.~Kim, M.~Kojima, and M.~Muramatsu}, {\em Sums of squares and
  semidefinite program relaxations for polynomial optimization problems with
  structured sparsity}, SIAM Journal on Optimization, 17 (2006), pp.~218--242.

\bibitem{waki2013facial}
{\sc H.~Waki and M.~Muramatsu}, {\em Facial reduction algorithms for conic
  optimization problems}, Journal of Optimization Theory and Applications, 158
  (2013), pp.~188--215.

\bibitem{wiegele2022sdp}
{\sc A.~Wiegele and S.~Zhao}, {\em Sdp-based bounds for graph partition via
  extended admm}, Computational Optimization and Applications, 82 (2022),
  pp.~251--291.

\bibitem{wolkowicz2012handbook}
{\sc H.~Wolkowicz, R.~Saigal, and L.~Vandenberghe}, {\em Handbook of
  semidefinite programming: theory, algorithms, and applications}, vol.~27,
  Springer Science \& Business Media, 2012.

\bibitem{wolkowicz1999semidefinite}
{\sc H.~Wolkowicz and Q.~Zhao}, {\em Semidefinite programming relaxations for
  the graph partitioning problem}, Discrete Applied Mathematics, 96 (1999),
  pp.~461--479.

\bibitem{zhout1998effective}
{\sc G.~Zhout and M.~Gen}, {\em An effective genetic algorithm approach to the
  quadratic minimum spanning tree problem}, Computers \& Operations Research,
  25 (1998), pp.~229--237.

\bibitem{zhu2019sieve}
{\sc Y.~Zhu, G.~Pataki, and Q.~Tran-Dinh}, {\em Sieve-{SDP}: a simple facial
  reduction algorithm to preprocess semidefinite programs}, Mathematical
  Programming Computation, 11 (2019), pp.~503--586.

\end{thebibliography}
\addcontentsline{toc}{section}{Bibliography}

\newpage
\appendix

	\section{Illustrative Toy Example}\label{toy_example}
	Consider the following binary feasible set with two binary variables $x \in \{0,1\}^2$:
	\begin{equation}\label{toy}
		P = \{ x \in \{0,1\}^2 \mid x_1 + x_2 = 1 \}.
	\end{equation}
	The feasible set contains exactly two points: $v_1 = (1,0)$ and $v_2 = (0,1)$. Note that we do not introduce an objective function in this example, as it does not play any role in facial reduction.
	
	\paragraph{SDP Relaxation.}
	We consider the following SDP relaxation. Define the affine subspace
	$$ L := \{ Y \in \mathbb{S}^{3} \mid \operatorname{arrow}(Y) = e_{0}, a^{T}Ya = 0 \}, $$
	where $a := (-1, 1, 1)^\top$ and $\operatorname{arrow}$ is defined in \eqref{def_arrow}. Then $L \cap \mathbb{S}_{+}^{3}$ is an SDP relaxation for \eqref{toy}, i.e., $S \subset L \cap \mathbb{S}_{+}^{3}$, where $S$ is the lifted feasible set defined in \eqref{liftedP}. 
	
	It is straightforward to verify that this SDP relaxation \textbf{does not satisfy Slater's condition}, as $a^{T}Ya > 0$ for any positive definite $Y$. The standard FRA requires us to find an exposing vector from the following set:
	$$ L^{\perp} \cap \mathbb{S}_{+}^{3} = \{ \operatorname{arrow}^{*}(y) + zaa^{T} \in \mathbb{S}_{+}^{3} \mid e_{0}^{T}y = 0 , y \in \mathbb{R}^{3}, z \in \mathbb{R} \}. $$
	In general, this amounts to solving an SDP feasibility problem with a matrix variable of order $3$. 
	
	We clarify the behavior of the primal FRA depending on the number of binary feasible solutions generated by Algorithm~\ref{alg3}:
	\begin{enumerate}
		\item \textbf{Algorithm~\ref{alg3} generates both feasible solutions $v_1 = (1,0)$ and $v_2 = (0,1)$ and terminates at Line 7.} In this case, we immediately conclude that the minimal face of $\mathbb{S}_{+}^{3}$ containing $S$ is given by 
		$$ F := \left\{ Y \in \mathbb{S}_{+}^{3} \mid \operatorname{range}(Y) \subseteq \operatorname{span}\left\{ \begin{bmatrix} 1\\ v_{1} \end{bmatrix}, \begin{bmatrix} 1\\ v_{2} \end{bmatrix} \right\} \right\}. $$
		The facially reduced formulation is $L \cap F$, which is strictly feasible, and the order of its matrix variable is exactly $2$. We emphasize that we do not need to solve an SDP auxiliary problem to perform this reduction.
		
		\item \textbf{Algorithm~\ref{alg3} generates a single feasible solution $v_1 = (1,0)$ and terminates at Line 9.} Define $G$ as in \eqref{newG}. Then we have the complementary face:
		$$ G^{\triangle} = \left\{ W \in \mathbb{S}_+^3 \mid \operatorname{range}(W) \subseteq \operatorname{span}\left\{ \begin{bmatrix} 1\\ -1\\ 0 \end{bmatrix}, \begin{bmatrix} 0\\ 0\\ 1 \end{bmatrix} \right\} \right\}. $$
		Now we can simplify the FR auxiliary problem $L^{\perp} \cap \mathbb{S}_{+}^{3}$ by finding an appropriate feasible solution in $L^{\perp} \cap G^{\triangle}$. The latter is an SDP feasibility problem where the order of the matrix variable is essentially $2$, which is one less than that of the original FR auxiliary problem. 
		
		\item \textbf{Algorithm~\ref{alg3} fails to generate any feasible solutions.} In this worst-case scenario, the primal FRA fails and provides no reduction.
	\end{enumerate}

	\section{The detail numerical results}\label{num_details}

	Tables~\ref{tab:dnn} to~\ref{tab:vshor2} present detailed experimental results. Specifically, Table~\ref{tab:dnn} reports results for the DNN relaxation, while Table~\ref{tab:shor} covers the standard Shor's relaxation; both relaxations lift binary and continuous variables into a single matrix variable. In contrast, Tables~\ref{tab:vshor1} and~\ref{tab:vshor2} present results for the variant of Shor's relaxation defined in \eqref{bqp_bq}, which lifts only the binary variables, yielding a significantly smaller SDP relaxation.

	For each instance, we report the following information:
	
	\begin{enumerate}
		\item \textbf{Aux. Size}: The order of the matrix variable in the FR auxiliary problem. For the standard FRA, this is equal to the matrix variable order in the original SDP relaxation. For the primal FRA, it corresponds to the matrix variable order in the reduced FR auxiliary problem. A value of zero indicates that the primal FRA terminated at Line~3 of Algorithm~\ref{alg2}, restoring Slater's condition without solving the FR auxiliary problem.
		
		\item \textbf{Time}: The total time in seconds required to perform facial reduction. For the standard FRA, this is the time to solve the FR auxiliary problem. For the primal FRA, we report two components: $T_1$ (time to construct a feasible solution) and $T_2$ (time to solve the reduced FR auxiliary problem).
		
		\item \textbf{Solver Status}: The status returned by MOSEK. \emph{Success} indicates that the SDP was solved correctly. \emph{Stall} denotes limited progress due to numerical issues, although the solution may still be usable. \emph{Fail} indicates that a reliable solution could not be obtained. \emph{Out of memory} means the problem could not be solved due to insufficient system memory.
		
		\item \textbf{Slater Condition}: \emph{Yes} indicates that the facially reduced formulation satisfies Slater’s condition. \emph{No} indicates that Slater’s condition was not verified. \rvv{We generate affinely independent feasible points one at a time by solving the randomized subproblem of Lemma~\ref{quad}, namely $\min/\max\{u^{\top} x \mid x \in P\}$ for a random direction $u \in H^{\perp}$. A candidate $x$ is accepted as affinely independent from the current points $x_0,\dots,x_k$ (with reference point $v$) when $|u^{\top}(x - v)| > 10^{-2}$, the numerical version of the test $u^{\top}(x-v)\neq 0$. When a maximal set of such points is obtained, the reduced face is minimal and Slater's condition holds; we record \emph{Yes}, solve no auxiliary problem, and the \textbf{Aux.~Size} entry is zero. When maximality cannot be certified, we solve the FR auxiliary problem $L^{\perp}\cap G^{\triangle}$, in which the exposing vector is normalized to have unit trace. If MOSEK certifies this problem infeasible, no exposing vector exists, Slater's condition holds, and we record \emph{Yes}. If instead a feasible exposing vector is found, the face can be reduced further but strict feasibility is not certified in this single step, and we record \emph{No}; the reported rank of this exposing vector is obtained by counting eigenvalues above $10^{-5}$, but the \emph{Yes}/\emph{No} verdict itself relies only on MOSEK's infeasibility certificate. When no reliable solution is available, due to solver failure or insufficient memory, the entry is marked \emph{N.A.} A stalled run (status \texttt{MSK\_RES\_TRM\_STALL}) is still classified by the rules above, using the problem status MOSEK returns, with no additional verification; in our experiments no stalled solve produced a \emph{No}. Of the $145$ instances for which the primal FRA restored Slater's condition, $115$ were certified from sampling alone and the remaining $30$ via infeasibility of the reduced auxiliary problem.}
	\end{enumerate}

	The results demonstrate that the primal FRA substantially improves computational efficiency over the standard FRA. In nearly all cases, the size of the FR auxiliary problem is significantly reduced under the primal FRA, indicating more effective facial reduction. Moreover, the total runtime is consistently lower, with the overhead of the primal FRA (i.e., computing a feasible solution) being negligible relative to solving the reduced FR auxiliary problem.
	
	A further key advantage is improved numerical robustness. Several instances that failed under the standard FRA due to numerical instability were successfully handled by the primal FRA. In particularly challenging cases where the standard FRA did not restore Slater’s condition, the primal FRA succeeded, improving solver reliability and ensuring strong duality.
	
	Overall, the empirical evidence strongly supports the effectiveness of the primal FRA in reducing auxiliary problem complexity, enhancing numerical stability, and restoring Slater’s condition—while incurring minimal computational overhead.

	\begin{landscape}

	\begin{table}[H]
		\centering
		\small
		\begin{tabular}{c|cccc|ccccccc} \hline
			\multirow{2}{*}{Name} & \multicolumn{4}{c|}{The standard FRA} &  \multicolumn{6}{c}{The primal FRA}  \\  \cline{2-11}
			& Aux.  Size & Time & Solver Status& Slater condition &Aux. Size  & $T_{1}$ & $T_{2}$ & Time & Solver Status& Slater condition   \\\hline
			gr4x6                 &    49        &    0.25      &    Success         &    No          &    0         &    0.21      &    0.00      &    0.21      &    Success         &    Yes         \\ 
			markshare-5-0         &    46        &    0.28      &    Success         &    No          &    0         &    0.28      &    0.00      &    0.28      &    Success         &    Yes         \\ 
			markshare-4-0         &    35        &    0.36      &    Success         &    No          &    0         &    1.18      &    0.00      &    1.18      &    Success         &    Yes         \\ 
			markshare1            &    63        &    1.08      &    Success         &    No          &    0         &    5.40      &    0.00      &    5.40      &    Success         &    Yes         \\ 
			neos5                 &    64        &    1.21      &    Success         &    Yes         &    0         &    0.48      &    0.00      &    0.48      &    Success         &    Yes         \\ 
			markshare2            &    75        &    2.23      &    Success         &    No          &    0         &    0.48      &    0.00      &    0.48      &    Success         &    Yes         \\ 
			pk1                   &    87        &    3.23      &    Success         &    No          &    0         &    0.86      &    0.00      &    0.86      &    Success         &    Yes         \\ 
			b-ball                &    101       &    5.61      &    Success         &    No          &    0         &    0.56      &    0.00      &    0.56      &    Success         &    Yes         \\ 
			neos-5192052-neckar   &    181       &    51.29     &    Success         &    No          &    0         &    2.03      &    0.00      &    2.03      &    Success         &    Yes         \\ 
			v150d30-2hopcds       &    151       &    54.32     &    Success         &    No          &    0         &    11.36     &    0.00      &    11.36     &    Success         &    Yes         \\ 
			mas74                 &    152       &    112.76    &    Success         &    Yes         &    0         &    6.02      &    0.00      &    6.02      &    Success         &    Yes         \\ 
			mas76                 &    152       &    122.36    &    Success         &    Yes         &    0         &    2.45      &    0.00      &    2.45      &    Success         &    Yes         \\ 
			gsvm2rl3              &    242       &    158.02    &    Success         &    Yes         &    0         &    6.95      &    0.00      &    6.95      &    Success         &    Yes         \\ 
			assign1-5-8           &    157       &    171.03    &    Success         &    No          &    0         &    4.81      &    0.00      &    4.81      &    Success         &    Yes         \\ 
			2club200v15p5scn      &    201       &    192.38    &    Success         &    Yes         &    0         &    12.77     &    0.00      &    12.77     &    Success         &    Yes         \\ 
			neos-5140963-mincio   &    197       &    250.46    &    Success         &    No          &    28        &    114.24    &    4.44      &    118.68    &    Success         &    No          \\ 
			glass-sc              &    215       &    320.04    &    Success         &    Yes         &    0         &    13.71     &    0.00      &    13.71     &    Success         &    Yes         \\ 
			iis-glass-cov         &    215       &    330.77    &    Success         &    Yes         &    0         &    13.50     &    0.00      &    13.50     &    Success         &    Yes         \\ 
			mad                   &    221       &    479.20    &    Success         &    No          &    0         &    3.16      &    0.00      &    3.16      &    Success         &    Yes         \\ 
			neos-3754480-nidda    &    254       &    481.36    &    Success         &    Yes         &    0         &    41.20     &    0.00      &    41.20     &    Success         &    Yes         \\ 
			p0201                 &    202       &    498.13    &    Failed          &    N.A.        &    62        &    3.29      &    32.25     &    35.53     &    Failed          &    N.A.        \\ 
			prod1                 &    251       &    1027.85   &    Success         &    No          &    44        &    28.77     &    6.54      &    35.31     &    Success         &    No          \\ 
			prod2                 &    302       &    2117.50   &    Success         &    No          &    39        &    60.37     &    8.64      &    69.01     &    Success         &    No          \\ 
			supportcase14         &    305       &    4509.23   &    Success         &    No          &    201       &    7.15      &    548.62    &    555.77    &    Success         &    No          \\ 
			\hline 
		\end{tabular}
		\caption{Algorithm performance on the DNN relaxation.}
		\label{tab:dnn}
	\end{table}
	\begin{table}[H]
		\centering
		\small
		\scalebox{0.86}{%
		\begin{tabular}{c|cccc|ccccccc} \hline
			\multirow{2}{*}{Name} & \multicolumn{4}{c|}{The standard FRA} &  \multicolumn{6}{c}{The primal FRA}  \\  \cline{2-11}
			&Aux. Size & Time & Solver Status& Slater condition & Aux. Size  & $T_{1}$ & $T_{2}$ & Time & Solver Status& Slater condition   \\\hline
			markshare-4-0         &    35        &    0.06      &    Success         &    Yes         &    0         &    1.21      &    0.00      &    1.21      &    Success         &    Yes         \\ 
			markshare-5-0         &    46        &    0.09      &    Success         &    Yes         &    0         &    0.32      &    0.00      &    0.32      &    Success         &    Yes         \\ 
			gr4x6                 &    49        &    0.13      &    Success         &    Yes         &    0         &    0.24      &    0.00      &    0.24      &    Success         &    Yes         \\ 
			neos5                 &    64        &    0.29      &    Success         &    Yes         &    0         &    0.48      &    0.00      &    0.48      &    Success         &    Yes         \\ 
			markshare1            &    63        &    0.30      &    Success         &    Yes         &    0         &    5.43      &    0.00      &    5.43      &    Success         &    Yes         \\ 
			markshare2            &    75        &    0.57      &    Success         &    Yes         &    0         &    0.49      &    0.00      &    0.49      &    Success         &    Yes         \\ 
			pk1                   &    87        &    1.27      &    Success         &    Yes         &    0         &    0.88      &    0.00      &    0.88      &    Success         &    Yes         \\ 
			b-ball                &    101       &    2.74      &    Success         &    Yes         &    0         &    0.59      &    0.00      &    0.59      &    Success         &    Yes         \\ 
			mas76                 &    152       &    12.16     &    Success         &    Yes         &    0         &    2.49      &    0.00      &    2.49      &    Success         &    Yes         \\ 
			v150d30-2hopcds       &    151       &    13.33     &    Success         &    No          &    0         &    11.37     &    0.00      &    11.37     &    Success         &    Yes         \\ 
			mas74                 &    152       &    13.39     &    Success         &    Yes         &    0         &    6.00      &    0.00      &    6.00      &    Success         &    Yes         \\ 
			assign1-5-8           &    157       &    25.12     &    Stall           &    Yes         &    0         &    4.94      &    0.00      &    4.94      &    Success         &    Yes         \\ 
			neos-5192052-neckar   &    181       &    30.52     &    Success         &    Yes         &    0         &    2.03      &    0.00      &    2.03      &    Success         &    Yes         \\ 
			neos-5140963-mincio   &    197       &    42.30     &    Success         &    Yes         &    28        &    113.99    &    1.26      &    115.25    &    Success         &    Yes         \\ 
			p0201                 &    202       &    43.02     &    Success         &    Yes         &    62        &    2.90      &    7.83      &    10.73     &    Success         &    Yes         \\ 
			2club200v15p5scn      &    201       &    53.26     &    Success         &    Yes         &    0         &    12.74     &    0.00      &    12.74     &    Success         &    Yes         \\ 
			iis-glass-cov         &    215       &    58.69     &    Success         &    Yes         &    0         &    13.53     &    0.00      &    13.53     &    Success         &    Yes         \\ 
			glass-sc              &    215       &    67.97     &    Success         &    Yes         &    0         &    13.91     &    0.00      &    13.91     &    Success         &    Yes         \\ 
			mad                   &    221       &    72.56     &    Success         &    Yes         &    0         &    2.67      &    0.00      &    2.67      &    Success         &    Yes         \\ 
			gsvm2rl3              &    242       &    124.86    &    Success         &    Yes         &    0         &    7.29      &    0.00      &    7.29      &    Success         &    Yes         \\ 
			prod1                 &    251       &    155.15    &    Success         &    Yes         &    44        &    29.21     &    2.34      &    31.55     &    Failed          &    N.A.        \\ 
			supportcase14         &    305       &    354.02    &    Success         &    No          &    201       &    7.55      &    116.64    &    124.19    &    Success         &    No          \\ 
			neos-3754480-nidda    &    254       &    415.66    &    Failed          &    N.A.        &    0         &    41.66     &    0.00      &    41.66     &    Success         &    Yes         \\ 
			prod2                 &    302       &    443.06    &    Success         &    Yes         &    39        &    60.52     &    2.43      &    62.95     &    Success         &    Yes         \\ 
			supportcase16         &    320       &    464.48    &    Success         &    No          &    217       &    8.53      &    137.82    &    146.35    &    Success         &    No          \\ 
			iis-hc-cov            &    298       &    475.61    &    Failed          &    N.A.        &    0         &    18.63     &    0.00      &    18.63     &    Success         &    Yes         \\ 
			neos-1430701          &    313       &    569.62    &    Success         &    Yes         &    0         &    5.14      &    0.00      &    5.14      &    Success         &    Yes         \\ 
			probportfolio         &    321       &    675.83    &    Success         &    Yes         &    276       &    102.37    &    255.10    &    357.47    &    Success         &    Yes         \\ 
			ran13x13              &    339       &    992.48    &    Success         &    Yes         &    0         &    9.76      &    0.00      &    9.76      &    Success         &    Yes         \\ 
			control30-3-2-3       &    333       &    1250.47   &    Failed          &    N.A.        &    171       &    345.81    &    338.36    &    684.17    &    Stall           &    Yes         \\ 
			pigeon-08             &    345       &    1378.78   &    Success         &    Yes         &    152       &    231.30    &    47.44     &    278.74    &    Success         &    Yes         \\ 
			nsa                   &    389       &    N.A.      &    Out of Memory   &    N.A.        &    96        &    123.16    &    27.10     &    150.27    &    Failed          &    N.A.        \\ 
			gsvm2rl5              &    402       &    N.A.      &    Out of Memory   &    N.A.        &    0         &    41.24     &    0.00      &    41.24     &    Success         &    Yes         \\ 
			neos-3611689-kaihu    &    422       &    N.A.      &    Out of Memory   &    N.A.        &    0         &    14.37     &    0.00      &    14.37     &    Success         &    Yes         \\ 
			neos-3610040-iskar    &    431       &    N.A.      &    Out of Memory   &    N.A.        &    0         &    13.51     &    0.00      &    13.51     &    Success         &    Yes         \\ 
			supportcase26         &    437       &    N.A.      &    Out of Memory   &    N.A.        &    115       &    441.93    &    200.76    &    642.69    &    Success         &    Yes         \\ 
			rlp1                  &    462       &    N.A.      &    Out of Memory   &    N.A.        &    0         &    12.09     &    0.00      &    12.09     &    Success         &    Yes         \\ 
			neos-3611447-jijia    &    473       &    N.A.      &    Out of Memory   &    N.A.        &    0         &    17.70     &    0.00      &    17.70     &    Success         &    Yes         \\ 
			k16x240b              &    481       &    N.A.      &    Out of Memory   &    N.A.        &    0         &    17.60     &    0.00      &    17.60     &    Success         &    Yes         \\ 
			nexp-50-20-1-1        &    491       &    N.A.      &    Out of Memory   &    N.A.        &    0         &    21.41     &    0.00      &    21.41     &    Success         &    Yes         \\ 
			pigeon-10             &    491       &    N.A.      &    Out of Memory   &    N.A.        &    230       &    480.00    &    229.21    &    709.21    &    Success         &    Yes         \\ 
			ran12x21              &    505       &    N.A.      &    Out of Memory   &    N.A.        &    0         &    22.61     &    0.00      &    22.61     &    Success         &    Yes         \\ 
			ran14x18-disj-8       &    505       &    N.A.      &    Out of Memory   &    N.A.        &    0         &    46.34     &    0.00      &    46.34     &    Success         &    Yes         \\ 
			\hline 
		\end{tabular}
		}
		\caption{Algorithm performance on the Shor's relaxation.}
		\label{tab:shor}
	\end{table}
	\begin{table}[H]
		\centering
		\small
		\scalebox{0.78}{%
		\begin{tabular}{c|cccc|ccccccc} \hline
			\multirow{2}{*}{Name} & \multicolumn{4}{c|}{The standard FRA} &  \multicolumn{6}{c}{The primal FRA}  \\  \cline{2-11}
			& Aux.  Size & Time & Solver Status& Slater condition & Aux. Size  & $T_{1}$ & $T_{2}$ & Time & Solver Status& Slater condition   \\\hline
			neos-5192052-neckar   &    25        &    0.02      &    Success         &    Yes         &    0         &    0.09      &    0.00      &    0.09      &    Success         &    Yes         \\ 
			gr4x6                 &    25        &    0.03      &    Success         &    Yes         &    0         &    0.21      &    0.00      &    0.21      &    Success         &    Yes         \\ 
			markshare-4-0         &    31        &    0.03      &    Success         &    Yes         &    0         &    0.11      &    0.00      &    0.11      &    Success         &    Yes         \\ 
			bienst1               &    29        &    0.04      &    Success         &    Yes         &    0         &    0.18      &    0.00      &    0.18      &    Success         &    Yes         \\ 
			markshare-5-0         &    41        &    0.07      &    Success         &    Yes         &    0         &    0.17      &    0.00      &    0.17      &    Success         &    Yes         \\ 
			bienst2               &    36        &    0.07      &    Success         &    Yes         &    0         &    0.21      &    0.00      &    0.21      &    Success         &    Yes         \\ 
			nsa                   &    37        &    0.08      &    Success         &    Yes         &    0         &    0.41      &    0.00      &    0.41      &    Success         &    Yes         \\ 
			markshare1            &    51        &    0.12      &    Success         &    Yes         &    0         &    0.22      &    0.00      &    0.22      &    Success         &    Yes         \\ 
			osorio-cta            &    8         &    0.14      &    Success         &    Yes         &    0         &    0.29      &    0.00      &    0.29      &    Success         &    Yes         \\ 
			neos5                 &    54        &    0.15      &    Success         &    Yes         &    0         &    0.37      &    0.00      &    0.37      &    Success         &    Yes         \\ 
			pk1                   &    56        &    0.17      &    Success         &    Yes         &    0         &    0.24      &    0.00      &    0.24      &    Success         &    Yes         \\ 
			qiu                   &    49        &    0.17      &    Success         &    Yes         &    0         &    0.60      &    0.00      &    0.60      &    Success         &    Yes         \\ 
			newdano               &    57        &    0.24      &    Success         &    Yes         &    0         &    0.31      &    0.00      &    0.31      &    Success         &    Yes         \\ 
			markshare2            &    61        &    0.25      &    Success         &    Yes         &    0         &    0.29      &    0.00      &    0.29      &    Success         &    Yes         \\ 
			gsvm2rl3              &    61        &    0.28      &    Success         &    Yes         &    0         &    0.81      &    0.00      &    0.81      &    Success         &    Yes         \\ 
			fastxgemm-n2r7s4t1    &    57        &    0.30      &    Failed          &    N.A.        &    0         &    7.21      &    0.00      &    7.21      &    Success         &    Yes         \\ 
			neos-3754480-nidda    &    51        &    0.34      &    Stall           &    Yes         &    0         &    0.28      &    0.00      &    0.28      &    Success         &    Yes         \\ 
			fastxgemm-n2r6s0t2    &    49        &    0.42      &    Failed          &    N.A.        &    0         &    3.27      &    0.00      &    3.27      &    Success         &    Yes         \\ 
			istanbul-no-cutoff    &    31        &    0.45      &    Success         &    Yes         &    1         &    25.42     &    0.15      &    25.58     &    Success         &    Yes         \\ 
			neos-2978193-inde     &    65        &    0.50      &    Success         &    Yes         &    0         &    1.76      &    0.00      &    1.76      &    Success         &    Yes         \\ 
			danoint               &    57        &    0.55      &    Success         &    Yes         &    0         &    46.53     &    0.00      &    46.53     &    Success         &    Yes         \\ 
			neos-3610051-istra    &    77        &    0.83      &    Success         &    Yes         &    0         &    0.83      &    0.00      &    0.83      &    Success         &    Yes         \\ 
			neos-3610173-itata    &    78        &    0.96      &    Success         &    Yes         &    0         &    0.64      &    0.00      &    0.64      &    Success         &    Yes         \\ 
			neos-5100895-inster   &    57        &    0.97      &    Failed          &    N.A.        &    34        &    61.82     &    0.74      &    62.57     &    Stall           &    Yes         \\ 
			dcmulti               &    76        &    1.07      &    Stall           &    Yes         &    6         &    1.08      &    0.02      &    1.10      &    Stall           &    Yes         \\ 
			neos-807639           &    81        &    1.25      &    Success         &    Yes         &    51        &    10.28     &    0.37      &    10.65     &    Success         &    Yes         \\ 
			neos-3610040-iskar    &    86        &    1.26      &    Success         &    Yes         &    0         &    0.78      &    0.00      &    0.78      &    Success         &    Yes         \\ 
			neos-3611447-jijia    &    86        &    1.31      &    Success         &    Yes         &    0         &    0.74      &    0.00      &    0.74      &    Success         &    Yes         \\ 
			b-ball                &    89        &    1.32      &    Success         &    Yes         &    0         &    0.60      &    0.00      &    0.60      &    Success         &    Yes         \\ 
			neos-3611689-kaihu    &    89        &    1.57      &    Success         &    Yes         &    0         &    0.80      &    0.00      &    0.80      &    Success         &    Yes         \\ 
			gsvm2rl5              &    101       &    2.04      &    Success         &    Yes         &    0         &    2.20      &    0.00      &    2.20      &    Success         &    Yes         \\ 
			rmatr100-p10          &    101       &    2.41      &    Success         &    Yes         &    0         &    2.99      &    0.00      &    2.99      &    Success         &    Yes         \\ 
			rmatr100-p5           &    101       &    2.52      &    Success         &    Yes         &    0         &    4.43      &    0.00      &    4.43      &    Success         &    Yes         \\ 
			pg                    &    101       &    2.61      &    Success         &    Yes         &    0         &    3.33      &    0.00      &    3.33      &    Success         &    Yes         \\ 
			pg5-34                &    101       &    2.63      &    Success         &    Yes         &    0         &    1.18      &    0.00      &    1.18      &    Success         &    Yes         \\ 
			control30-3-2-3       &    91        &    2.79      &    Failed          &    N.A.        &    0         &    0.94      &    0.00      &    0.94      &    Success         &    Yes         \\ 
			mod011                &    97        &    2.92      &    Stall           &    Yes         &    0         &    4.84      &    0.00      &    4.84      &    Success         &    Yes         \\ 
			app3                  &    101       &    3.79      &    Stall           &    Yes         &    4         &    3.93      &    0.04      &    3.97      &    Stall           &    Yes         \\ 
			milo-v13-4-3d-3-0     &    121       &    5.70      &    Success         &    Yes         &    94        &    130.06    &    10.84     &    140.90    &    Success         &    Yes         \\ 
			assign1-5-8           &    131       &    6.24      &    Success         &    Yes         &    0         &    0.86      &    0.00      &    0.86      &    Success         &    Yes         \\ 
			neos-1396125          &    130       &    7.52      &    Success         &    Yes         &    45        &    132.02    &    5.60      &    137.62    &    Success         &    Yes         \\ 
			neos-3660371-kurow    &    145       &    9.08      &    Success         &    Yes         &    44        &    44.62     &    1.89      &    46.50     &    Success         &    Yes         \\ 
			mas74                 &    151       &    10.76     &    Success         &    Yes         &    0         &    0.86      &    0.00      &    0.86      &    Success         &    Yes         \\ 
			v150d30-2hopcds       &    151       &    12.77     &    Success         &    No          &    0         &    12.48     &    0.00      &    12.48     &    Success         &    Yes         \\ 
			prod1                 &    150       &    12.77     &    Success         &    Yes         &    25        &    9.69      &    0.21      &    9.89      &    Success         &    Yes         \\ 
			\hline 
		\end{tabular}
		}
		\caption{Algorithm performance on the variant of Shor’s relaxation (1).}
		\label{tab:vshor1}
	\end{table}

	\begin{table}[H]
		\centering
		\small
		\scalebox{0.78}{%
		\begin{tabular}{c|cccc|ccccccc} \hline
			\multirow{2}{*}{Name} & \multicolumn{4}{c|}{The standard FRA} &  \multicolumn{6}{c}{The primal FRA}  \\  \cline{2-11}
			& Aux. Size & Time & Solver Status& Slater condition &Aux. Size  & $T_{1}$ & $T_{2}$ & Time & Solver Status& Slater condition   \\\hline
			mas76                 &    151       &    12.87     &    Success         &    Yes         &    0         &    0.91      &    0.00      &    0.91      &    Success         &    Yes         \\ 
			neos-1430701          &    157       &    15.42     &    Success         &    Yes         &    0         &    0.98      &    0.00      &    0.98      &    Success         &    Yes         \\ 
			milo-v13-4-3d-4-0     &    161       &    17.15     &    Success         &    Yes         &    127       &    170.50    &    33.47     &    203.97    &    Success         &    Yes         \\ 
			ran13x13              &    170       &    21.13     &    Success         &    Yes         &    0         &    1.66      &    0.00      &    1.66      &    Success         &    Yes         \\ 
			neos-5140963-mincio   &    184       &    28.53     &    Success         &    Yes         &    28        &    19.68     &    1.39      &    21.07     &    Success         &    Yes         \\ 
			binkar10-1            &    171       &    29.58     &    Success         &    Yes         &    0         &    6.02      &    0.00      &    6.02      &    Success         &    Yes         \\ 
			neos-480878           &    190       &    40.88     &    Success         &    Yes         &    0         &    4.06      &    0.00      &    4.06      &    Success         &    Yes         \\ 
			gsvm2rl9              &    201       &    43.17     &    Success         &    Yes         &    0         &    19.21     &    0.00      &    19.21     &    Success         &    Yes         \\ 
			neos-3372571-onahau   &    186       &    50.14     &    Success         &    Yes         &    75        &    129.53    &    2.55      &    132.08    &    Success         &    Yes         \\ 
			prod2                 &    201       &    50.34     &    Success         &    Yes         &    24        &    26.39     &    0.58      &    26.97     &    Success         &    Yes         \\ 
			2club200v15p5scn      &    201       &    51.92     &    Success         &    Yes         &    0         &    12.79     &    0.00      &    12.79     &    Success         &    Yes         \\ 
			a2c1s1                &    193       &    53.84     &    Success         &    Yes         &    0         &    51.63     &    0.00      &    51.63     &    Success         &    Yes         \\ 
			rmatr200-p20          &    201       &    60.34     &    Success         &    Yes         &    0         &    44.75     &    0.00      &    44.75     &    Success         &    Yes         \\ 
			p0201                 &    202       &    61.80     &    Stall           &    Yes         &    62        &    2.60      &    6.84      &    9.44      &    Stall           &    Yes         \\ 
			glass-sc              &    215       &    61.81     &    Success         &    Yes         &    0         &    13.73     &    0.00      &    13.73     &    Success         &    Yes         \\ 
			a1c1s1                &    193       &    63.23     &    Stall           &    Yes         &    0         &    98.01     &    0.00      &    98.01     &    Success         &    Yes         \\ 
			van                   &    193       &    63.39     &    Success         &    Yes         &    0         &    192.59    &    0.00      &    192.59    &    Success         &    Yes         \\ 
			mad                   &    201       &    64.73     &    Failed          &    N.A.        &    0         &    2.09      &    0.00      &    2.09      &    Success         &    Yes         \\ 
			iis-glass-cov         &    215       &    65.10     &    Success         &    Yes         &    0         &    13.49     &    0.00      &    13.49     &    Success         &    Yes         \\ 
			bg512142              &    241       &    107.76    &    Success         &    No          &    0         &    15.50     &    0.00      &    15.50     &    Success         &    Yes         \\ 
			k16x240b              &    241       &    111.30    &    Success         &    Yes         &    0         &    13.37     &    0.00      &    13.37     &    Success         &    Yes         \\ 
			exp-1-500-5-5         &    251       &    141.01    &    Success         &    Yes         &    0         &    4.73      &    0.00      &    4.73      &    Success         &    Yes         \\ 
			neos-2629914-sudost   &    257       &    144.15    &    Success         &    Yes         &    0         &    10.33     &    0.00      &    10.33     &    Success         &    Yes         \\ 
			ran12x21              &    253       &    146.42    &    Success         &    Yes         &    0         &    4.09      &    0.00      &    4.09      &    Success         &    Yes         \\ 
			nexp-50-20-1-1        &    246       &    149.11    &    Success         &    No          &    0         &    3.99      &    0.00      &    3.99      &    Success         &    Yes         \\ 
			ran14x18-disj-8       &    253       &    161.37    &    Success         &    Yes         &    0         &    6.32      &    0.00      &    6.32      &    Success         &    Yes         \\ 
			bc1                   &    253       &    199.75    &    Failed          &    N.A.        &    0         &    134.63    &    0.00      &    134.63    &    Success         &    Yes         \\ 
			pigeon-08             &    273       &    220.45    &    Success         &    Yes         &    152       &    20.96     &    26.41     &    47.37     &    Failed          &    N.A.        \\ 
			supportcase14         &    305       &    348.73    &    Success         &    No          &    201       &    7.23      &    116.14    &    123.37    &    Success         &    No          \\ 
			supportcase16         &    320       &    444.53    &    Success         &    No          &    217       &    8.05      &    139.63    &    147.69    &    Success         &    No          \\ 
			probportfolio         &    301       &    457.84    &    Success         &    Yes         &    165       &    278.27    &    44.37     &    322.64    &    Success         &    Yes         \\ 
			b2c1s1                &    289       &    466.11    &    Success         &    Yes         &    0         &    180.33    &    0.00      &    180.33    &    Success         &    Yes         \\ 
			neos17                &    301       &    493.37    &    Success         &    Yes         &    0         &    5.45      &    0.00      &    5.45      &    Success         &    Yes         \\ 
			iis-hc-cov            &    298       &    508.35    &    Failed          &    N.A.        &    0         &    18.16     &    0.00      &    18.16     &    Success         &    Yes         \\ 
			b1c1s1                &    289       &    525.80    &    Stall           &    Yes         &    0         &    169.96    &    0.00      &    169.96    &    Success         &    Yes         \\ 
			neos-3665875-lesum    &    321       &    597.45    &    Success         &    Yes         &    157       &    325.58    &    177.87    &    503.46    &    Success         &    Yes         \\ 
			sp150x300d            &    301       &    847.32    &    Stall           &    Yes         &    29        &    39.87     &    0.54      &    40.40     &    Stall           &    Yes         \\ 
			neos-5188808-nattai   &    289       &    963.47    &    Stall           &    Yes         &    188       &    292.83    &    550.53    &    843.37    &    Stall           &    Yes         \\ 
			r50x360               &    361       &    1000.67   &    Success         &    Yes         &    1         &    19.99     &    0.23      &    20.22     &    Success         &    Yes         \\ 
			tr12-30               &    361       &    1052.19   &    Success         &    Yes         &    8         &    32.87     &    0.29      &    33.16     &    Success         &    Yes         \\ 
			neos-1442119          &    365       &    1109.93   &    Success         &    Yes         &    0         &    7.24      &    0.00      &    7.24      &    Success         &    Yes         \\ 
			uct-subprob           &    380       &    2093.29   &    Stall           &    Yes         &    0         &    8.13      &    0.00      &    8.13      &    Success         &    Yes         \\ 
			supportcase26         &    397       &    N.A.      &    Out of Memory   &    N.A.        &    55        &    400.32    &    24.59     &    424.92    &    Success         &    Yes         \\ 
			pigeon-10             &    401       &    N.A.      &    Out of Memory   &    N.A.        &    230       &    177.04    &    103.54    &    280.58    &    Success         &    Yes         \\ 
			aflow30a              &    422       &    N.A.      &    Out of Memory   &    N.A.        &    0         &    87.06     &    0.00      &    87.06     &    Success         &    Yes         \\ 
			\hline 
		\end{tabular}
		}
		\caption{Algorithm performance on the variant of Shor’s relaxation (2).}
		\label{tab:vshor2}
	\end{table}
	
	\end{landscape}

	\section{Further Results on Solving the SDP Relaxation}\label{sdpnum}

	
	The primary focus of this work is the preprocessing stage—specifically, solving the FR auxiliary problem efficiently. The effectiveness of the proposed primal FRA in this regard has already been demonstrated in the previous section. In this subsection, we present supplementary numerical results comparing the DNN relaxation and its facially reduced formulation.
	
	While the advantages of solving facially reduced SDP relaxations are well established in the literature, this approach is not without criticism. Key concerns include potential loss of sparsity and the introduction of numerical issues during the reformulation process. These challenges often necessitate problem-specific strategies and fine-tuning to fully realize the benefits of facial reduction. Indeed, most numerical experiments on SDP methods are conducted on instances where SDP formulations are known to perform well and can be regularized into stable forms—e.g., the Max-Cut problem, maximum stable set, and quadratic assignment problem. A detailed treatment of such implementations for the benchmark instances in MIPLIB lies beyond the scope of this paper.
	
	Nevertheless, the supplementary results in Table~\ref{tab:sdp} empirically confirm that facial reduction enhances numerical stability in practice. When solving the original DNN relaxations, the solver encountered numerical difficulties or failed to return reliable solutions for \rvv{eight of the ten} instances. In contrast, the facially reduced DNN formulations exhibited greater robustness, successfully solving four of these cases. Although the problem size is reduced after facial reduction, the total runtime can sometimes increase due to the aforementioned loss of sparsity. Addressing this trade-off often requires customized solutions depending on the problem structure.
	
	Table~\ref{tab:sdp} also confirms that MBQP problems are significantly more expensive to solve than MILP problems. Thus, when applying the SDP approach to solve an MBQP, it is reasonable to rely on MILP solvers to help regularize the SDP relaxation. This is the key idea behind the primal FRA.

	Each row in Table~\ref{tab:sdp} contains the following information:
	
	\begin{enumerate}
		\item \textbf{Size}: the order of the matrix variable in the DNN relaxation.
		
		\item \textbf{L.B.}: the lower bound obtained from solving the relaxation. In theory, the original relaxation and its facially reduced counterpart are equivalent and should produce the same bound, assuming both are solved correctly. In practice, however, especially for the original formulation, failure to satisfy Slater's condition may impair bound quality.
		
		\item \textbf{Time}: the computational time required to solve the relaxation.
		
		\item \textbf{Solver Status}: solver status as reported by MOSEK, using the same categorization as in Table~\ref{tab:dnn}—\emph{Success}, \emph{Stall}, or \emph{Fail}.
		
		\item \textbf{U.B.}: the best upper bound obtained by Gurobi within a 600-second time limit.
		
		\item \textbf{Time (Gurobi)}: the computational time used by Gurobi.
		
		\item \textbf{Status (Gurobi)}: the termination status reported by Gurobi.
	\end{enumerate}

	\begin{table}[H]
		\centering
		\small
		\resizebox{\linewidth}{!}{%
		\begin{tabular}{c|cccc|cccc|ccc} \\\hline
			& \multicolumn{4}{c|}{DNN Relaxation} & \multicolumn{4}{c}{Facially Reduced DNN Relaxation} & \multicolumn{3}{|c}{Gurobi} \\ \cline{2-12}
			Name & Size & L.B. & Time & Status & Size & L.B. & Time & Status & U.B. & Time & Status \\ \hline 
			markshare-4-0         &    34        &    NA        &           0.14      &    Failed          &    31        &    30.46     &           1.33      &    Stall           &    65.85     &    91.29     &    Optimal   \\ 
			markshare-5-0         &    45        &    NA        &           0.22      &    Failed          &    41        &    NA        &           2.58      &    Failed          &    235.41    &    600.02    &    Time limit\\ 
			gr4x6                 &    48        &    2863.42   &           0.34      &    Stall           &    40        &    2863.83   &           0.77      &    Success         &    3022.17   &    1.39      &    Optimal   \\ 
			markshare1            &    62        &    NA        &           0.61      &    Failed          &    51        &    109.16    &           6.64      &    Stall           &    2097.72   &    924.47    &    Time limit\\ 
			markshare2            &    74        &    NA        &           1.76      &    Failed          &    61        &    178.57    &          32.84      &    Stall           &    9850.66   &    970.81    &    Time limit\\ 
			pk1                   &    86        &    NA        &           2.58      &    Failed          &    72        &    146.58    &          18.54      &    Success         &    4293.04   &    1005.76   &    Time limit\\ 
			b-ball                &    100       &    887.75    &          10.18      &    Stall           &    82        &    887.83    &          30.04      &    Success         &    986.56    &    978.20    &    Time limit\\ 
			assign1-5-8           &    156       &    2076.92   &          70.88      &    Stall           &    127       &    NA        &         683.04      &    Failed          &    16614.17  &    1031.83   &    Time limit\\ 
			neos-5192052-neckar   &    180       &    NA        &         100.61      &    Failed          &    161       &    NA        &         279.05      &    Failed          &    0.00      &    0.04      &    Optimal   \\ 
			neos-5140963-mincio   &    196       &    19.45     &         158.99      &    Stall           &    170       &    19.61     &        1782.07      &    Success         &    1546.53   &    2464.42   &    Time limit\\ 
			\hline 
		\end{tabular}
		}
		\caption{Comparison between the original DNN relaxation and its facially reduced formulation}
		\label{tab:sdp}
	\end{table}

\end{document}